\documentclass[12pt, reqno]{amsart}
\usepackage{amssymb, amsmath, amsthm}

\title[Rapid Evolution of Complex Cycles]{Rapid Evolution of Complex Limit Cycles}

\author[Nikolay Dimitrov]{}

\email{dimitrov@math.mcgill.ca}



\newtheorem{thm}{Theorem}
\newtheorem{lem}{Lemma}[section]
\newtheorem{prop}{Proposition}
\newtheorem{Def}{Definition}
\newtheorem{cor}{Corollary}[section]

\newcommand{\Atlas}{\mathcal{A}}
\newcommand{\Integers}{\mathbb{Z}}
\newcommand{\Naturals}{\mathbb{N}}
\newcommand{\CC}{\mathbb{C}^2}
\newcommand{\cptwo}{\mathbb{CP}^2}
\newcommand{\RR}{\mathbb{R}^2}
\newcommand{\cc}{\mathbb{C}}
\newcommand{\Band}{\mathbb{B}}
\newcommand{\Disc}{\mathbb{D}}

\newcommand{\Deltae}{\Delta_{\e}}

\newcommand{\deltaprime}{\delta^{\prime}}
\newcommand{\deltaprimehat}{\hat{\delta}^{\prime}}
\newcommand{\deltaprimee}{\delta^{\prime}_{\e}}

\newcommand{\deltahat}{\hat{\delta}}

\newcommand{\e}{\varepsilon}
\newcommand{\etamax}{\eta_{max}}

\newcommand{\eprime}{\e^{\prime}}
\newcommand{\estar}{\e^{*}}
\newcommand{\edoublestar}{\e^{**}}
\newcommand{\etaprime}{\eta^{\prime}}

\newcommand{\etatilde}{\tilde{\eta}}
\newcommand{\Gammatilde}{\tilde{\Gamma}}
\newcommand{\Gammahat}{\hat{\Gamma}}

\newcommand{\phie}{\varphi^{\e}}

\newcommand{\Aprime}{A^{\prime}}
\newcommand{\Ahat}{\hat{A}}
\newcommand{\Aprimehat}{\hat{A}^{\prime}}

\newcommand{\Apzeroprime}{A^{\prime}_{p_0}}
\newcommand{\Apzerotilde}{\tilde{A}_{p_0}}

\newcommand{\alphaqzeroe}{\alpha_{q_0,\e}}
\newcommand{\alphaqonee}{\alpha_{q_1,\e}}

\newcommand{\alphaqje}{\alpha_{q_j,\e}}

\newcommand{\Bdelta}{B_{\delta_0}}
\newcommand{\Bpzero}{B_{p_0}}
\newcommand{\Bdeltahat}{\hat{B}_{\delta_0}}
\newcommand{\betaqzeroe}{\beta_{q_0,\e}}
\newcommand{\betaqonee}{\beta_{q_1,\e}}
\newcommand{\betaqtwoe}{\beta_{q_2,\e}}
\newcommand{\betaqje}{\beta_{q_j,\e}}

\newcommand{\Cdelta}{C_{\delta_0}}
\newcommand{\Cdeltaprime}{C_{\delta_0}^{\prime}}
\newcommand{\Cdeltahat}{\hat{C}_{\delta_0}}
\newcommand{\Cdeltaprimehat}{\hat{C}_{\delta_0}^{\prime}}
\newcommand{\Cpzero}{C_{p_0}}
\newcommand{\Cpzeroprime}{C^{\prime}_{p_0}}
\newcommand{\Chat}{\hat{C}}
\newcommand{\Cprimehat}{\hat{C}^{\prime}}

\newcommand{\CAhat}{\hat{C}_A}

\newcommand{\dtimess}{\Disc \times S}
\newcommand{\Dgamma}{D_{\gamma}}

\newcommand{\Edelta}{E_{\delta_0}}
\newcommand{\Ecdelta}{E(C_{\delta_0})}

\newcommand{\EAprime}{E(A^{\prime})}

\newcommand{\Fol}{\mathcal{F}^{\e}}
\newcommand{\Foliation}{\mathcal{F}}

\newcommand{\Ftilde}{\tilde{F}}

\newcommand{\Lqzero}{L_{q_0}}
\newcommand{\Lqone}{L_{q_1}}
\newcommand{\Lqtwo}{L_{q_2}}
\newcommand{\Lqj}{L_{q_j}}

\newcommand{\Pdelta}{P_{\delta_0, \e}}
\newcommand{\Pmap}{P_{\delta_0,\e}}

\newcommand{\PmapD}{P_{D_{\gamma}(\delta_0),\e}}

\newcommand{\Pmapezero}{P_{\delta_0,\e_0}}

\newcommand{\Pmaphat}{\hat{P}_{\delta_0,\e}}
\newcommand{\Pmaphatzero}{\hat{P}_{\delta_0,\e_0}}
\newcommand{\PmaphatD}{\hat{P}_{D_{\gamma}(\delta_0),\e}}

\newcommand{\Pmaphateprime}{\hat{P}_{\delta_0,\e^{\prime}}}
\newcommand{\Pmaphatezero}{\hat{P}_{\delta_0,\e_0}}

\newcommand{\Pmaptilde}{\tilde{P}_{\delta_0,\e}}

\newcommand{\Pmaptildeeprime}{\tilde{P}_{\delta_0,\e^{\prime}}}

\newcommand{\Pmaptildeestar}{\tilde{P}_{\delta_0,\e^{**}}}

\newcommand{\Pem}{P_{\e}^{(m)}}
\newcommand{\Pemprime}{P_{\e^{\prime}}^{(m)}}
\newcommand{\Pemstar}{P_{\e^{**}}^{(m)}}
\newcommand{\Pe}{P_{\e}}

\newcommand{\Pipzero}{\Pi_{p_0}}
\newcommand{\pizero}{\pi^{(0)}}

\newcommand{\phiqzeroe}{\phi_{q_0,\e}}
\newcommand{\phiqonee}{\phi_{q_1,\e}}
\newcommand{\phiqtwoe}{\phi_{q_2,\e}}
\newcommand{\phiqje}{\phi_{q_j,\e}}

\newcommand{\barphiqzeroe}{\bar{\phi}_{q_0,\e}}
\newcommand{\barphiqonee}{\bar{\phi}_{q_1,\e}}
\newcommand{\barphiqtwoe}{\bar{\phi}_{q_2,\e}}
\newcommand{\barphiqje}{\bar{\phi}_{q_j,\e}}

\newcommand{\phiqzeroehat}{\hat{\phi}_{\hat{q}_0,\e}}

\newcommand{\phizoneehat}{\hat{\phi}_{z_1,\e}}

\newcommand{\phizoneeprimehat}{\hat{\phi}_{z_1,\e^{\prime}}}

\newcommand{\phixzeroetilde}{\tilde{\phi}_{x_0,\e}}

\newcommand{\phixstaretilde}{\tilde{\phi}_{x^{*},\e}}
\newcommand{\phixonestaretilde}{\tilde{\phi}_{x^{*}_1,\e}}

\newcommand{\third}{\frac{1}{3}}
\newcommand{\twothird}{\frac{2}{3}}

\newcommand{\Uqzero}{U_{q_0}}
\newcommand{\Uqone}{U_{q_1}}
\newcommand{\Uqtwo}{U_{q_2}}
\newcommand{\Uqj}{U_{q_j}}

\newcommand{\Uxzerotilde}{\tilde{U}_{x_0}}

\newcommand{\Uxonestartilde}{\tilde{U}_{x_1^*}}
\newcommand{\Uxonestartildeprime}{\tilde{U}_{x_1^*}^{\prime}}
\newcommand{\Uxstartilde}{\tilde{U}_{x^*}}

\newcommand{\Uqzerohat}{\hat{U}_{\hat{q}_0}}

\newcommand{\Uzonehat}{\hat{U}_{z_1}}

\newcommand{\Xdelta}{X_{\delta_0}}
\newcommand{\Xdeltaprime}{X_{\delta_0}^{\prime}}
\newcommand{\Xdeltahat}{\hat{X}_{\delta_0}}
\newcommand{\Xdeltaprimehat}{\hat{X}_{\delta_0}^{\prime}}

\newcommand{\Yprime}{Y^{\prime}}

\begin{document}

\maketitle

\centerline{\scshape Nikolay Dimitrov}
\medskip
{\footnotesize
% please put the address of the first author
 \centerline{Department of Mathematics and Statistics}
   \centerline{McGill University}
   \centerline{805 Sherbrooke W.}
   \centerline{Montreal, QC H3A 2K6, Canada}
 }

\begin{abstract}
The current article studies certain problems related to complex
cycles of holomorphic foliations with singularities in the complex
plane. We focus on the case when polynomial differential one-form
gives rise to a foliation by Riemann surfaces. In this setting, a
complex cycle is defined as a nontrivial element of the
fundamental group of a leaf from the foliation. Whenever the
polynomial foliation comes from a perturbation of an exact
one-form, one can introduce the notion of a multi-fold cycle. This
type of cycle has at least one representative that determines a
free homotopy class of loops in an open fibred subdomain of the
complex plane. The topology of this subdomain is closely related
to the exact one-form mentioned earlier. We introduce and study
the notion of multi-fold cycles of a close-to-integrable
polynomial foliation. We also explore how these cycles correspond
to periodic orbits of a certain Poincar\'e map associated with the
foliation. Finally, we discuss the tendency of a continuous family
of multi-fold limit cycles to escape from certain large open
domains in the complex plane as the foliation converges to its
integrable part.
\end{abstract}

%%%%%%%%%%%%%%%%%%%%%%%%%%%%%%%%%%%%%%%%%%%%%%%%%%%%%%%%%%%%%%%%%
%INTRODUCTION

\section{Introduction}

Limit cycles of planar polynomial vector fields have long been a
focus of extensive research. For instance, one of the major
problems in this area of dynamical systems is the famous Hilbert's
16th problem \cite{I02} asking about the number and the location
of the limit cycles of a polynomial vector field of degree n in
the plane. Since the original Hilbert's problem continues to be
very persistent, some simplifications have been considered as
well. Among them is the so called infinitesimal Hilbert's 16
problem \cite{I02}, \cite{IY} concerned with the number of limit
cycles that can bifurcate from periodic solutions of a polynomial
Hamiltonian planar system by a small polynomial perturbation.
Recently, an answer to this question has been given in an article
by Binyamini, Novikov and Yakovenko \cite{BNY}.

When studying a planar polynomial vector field, an extension to
the complex domain proves to be helpful, an idea that can be
attributed to Petrovskii and Landis \cite{PL55}, \cite{PL57}. In
this way a polynomial complex vector field is obtained and the
holomorphic curves tangent to it form a partition of the complex
plane by Riemann surfaces, called a \emph{polynomial complex
foliation with singularities}, or in short \emph{polynomial
complex foliation} \cite{I02}, \cite{IY}.

We are going to focus on polynomial perturbations of a polynomial
Hamiltonian system in $\CC.$ More precisely we consider the
complex line field
\begin{equation}\label{basic}
F^{\e} = \ker(dH + \e \omega)
\end{equation}
with a one-form $\omega = Adx + Bdy,$ where $A,B$ and $H \in
\cc[x,y]$ are polynomials with complex coefficients and $\e$ is a
small complex parameter. As mentioned earlier, the holomorphic
curves tangent to $F^{\e}$ form a foliation of Riemann surfaces in
$\CC$ further denoted by $\Fol(\CC).$ Notice, that in the real
case the phase curves of a planar vector field are topologically
either lines or circles, i.e. curves with either a trivial or a
non-trivial (isomorphic to $\Integers$) fundamental group. This
simple observation leads us to the definition of a marked complex
cycle.

%%%%%%%%%%%%%%%%%%%%%%%%%%%%%%%%%%%%%%%%%%%%%%%%%%%%%%%%%%%%%%%%%%%%%5
%DEFINITION OF A COMPLEX CYCLE

\begin{Def} \label{cycle}
A marked complex cycle of a complex foliation is a nontrivial
element of the fundamental group of a leaf from the foliation with
a marked base point.
\end{Def}

\noindent We denote a marked complex cycle by $(\Delta, q)$ where
$\Delta$ is the homotopy class of loops on the leaf, all passing
through the same base point $q$. In general, a real phase curve of
a polynomial vector field in $\RR$ extends to a Riemann surface
tangent to the vector field's complexification in $\CC.$ Thus, a
closed phase curve in $\RR$ defines a loop on the corresponding
complex leaf, giving rise to a nontrivial element from the
fundamental group of that leaf \cite{I02}. In other words, a real
closed phase curve is a marked complex cycle on its
complexification.

%%%%%%%%%%%%%%%%%%%%%%%%%%%%%%%%%%%%%%%%%%%%%%%%%%%%%%%%%%%%%%%%%%%%
%UNPERTURBED VERSUS PERTURBED FOLIATION

When $\e=0$ the foliation $\Foliation^0(\CC)$ consists of
algebraic leaves of the form $S_{u} = \{p \in \CC \,:\, H(p)=u\}$
embedded in $\CC.$ From now on, we are going to refer to
$\Foliation^0(\CC)$ as the \emph{integrable foliation} and to
$\Fol(\CC)$ as the \emph{perturbed foliation}. The idea is to
study the complex cycles of $\Fol(\CC)$ using our knowledge of
$\Foliation^0(\CC)$. %refer to the first picture

%CONSTRUCTION OF THE POINCAR\'E MAP

One of the most powerful tools for studying foliations and
continuous dynamical systems in general, is the Poincar\'e map
\cite{I02}, \cite{IY}.

To construct a \emph{Poincar\'e map} for the foliation
$\Fol(\CC),$ one can follow several steps. Start by choosing a
point $p_0$ on a leaf $S_{u_0}$ of $\Foliation^0(\CC)$ and a
nontrivial loop $\delta_0 \subset S_{u_0}$ with a base point
$p_0.$ Take a small enough complex segment $L$ passing through
$p_0$ and transverse to the leaves of $\Foliation^0(\CC).$
Consider a tubular neighborhood $A$ of $\delta_0$ on the surface
$S_{u_0}.$ A tubular neighborhood $N(A)$ of $A$ in $\CC$ is
diffeomorphic to a direct product $A \times \Disc,$ where $\Disc
\subset \cc$ is the unit disc. Let $\varrho $ be the projection of
$N(A)$ onto $A$ along $\Disc.$ The direct product structure on
$N(A)$ can be chosen so that $L=\varrho^{-1}(p_0).$ If $\e$ is
chosen small enough, then for any point $q \in L$ close to $p_0$
the loop $\delta_0$ can be lifted to a curve $\delta(q)$ on the
leaf of the perturbed foliation $\Fol(\CC)$ passing through $q,$
so that $\delta(q)$ covers $\delta_0$ under the projection
$\varrho.$ By construction, $\delta(q)$ will have both of its
endpoints on $L,$ where $q \in L$ is one of them. Denote the
second endpoint by $P_{\delta_0, \e}(q) \in L$. Thus, we obtain a
correspondence $P_{\delta_0, \e} : L^{'} \to L$ where the open set
$L^{'} \subset L$ is the domain of $P_{\delta_0, \e}$. The map
$\Pdelta$ is holomorphic and close to identity. Notice that by
construction, if $\delta_0$ is homotopic on $S_{u_0}$ to another
loop $\delta_0^{'}$ passing through $p_0,$ then for small enough
$\e$ the two maps $P_{\delta_0, \e}$ and $P_{\delta^{'}_0, \e}$
will be equal.

%%%%%%%%%%%%%%%%%%%%%%%%%%%%%%%%%%%%%%%%%%%%%%%%%%%%%%%%%%%%%%%%
% CORRESPONDENCE BETWEEN COMPLEX CYCLES AND PERIODIC ORBITS

The Poincar\'e map described above has the property that if two
points from the cross-section $L$ are in the same orbit of the map
then they belong to the same leaf of the foliation. Moreover, a
marked complex cycle of $\Fol(\CC),$ with a base point on $L^{'}$
and a representative that covers $m$ times the loop $\delta_0$
under the projection $\varrho,$ gives rise to an $m-$periodic
orbit of $P_{\delta_0, \e}.$ The inverse is also true \cite{PL55},
\cite{PL57}. An $m$-periodic orbit corresponds to a marked complex
cycles of $\Fol(\CC)$ with a representative contained in $N(A),$
covering $\delta_0$ under the projection $\varrho$ a number of $m$
times. Notice that since $\varrho$ is a deformation retraction of
$N(a)$ onto $A,$ the representative will be free homotopic to
$\delta_0^m$ inside $N(A).$

%%%%%%%%%%%%%%%%%%%%%%%%%%%%%%%%%%%%%%%%%%%%%%%%%%%%%%%%%%%%%%
%DEFINITION OF A MULTI-FOLD CYCLES

\begin{Def} \label{mfold}
A marked complex cycle is called a $\delta_0,m-$fold cycle
provided that it gives rise to an $m-$periodic orbit of some
Poincar\'e map $\Pdelta.$
\end{Def}

When $m>1$ and we do not want to specify the characteristics
$\delta_0$ and $m$ we call such a cycle \emph{multi-fold}. For any
$m>0,$ it is not difficult to see that $\Pdelta^m = P_{\delta_0^m,
\e}.$ Then a $\delta_0,m-$fold cycle is represented by a fixed
point of the $m-$th iteration of $\Pdelta$ or equivalently by a
fixed point of $P_{\delta_0^m, \e}.$ Now, we can give a definition
for a marked limit cycle.

%%%%%%%%%%%%%%%%%%%%%%%%%%%%%%%%%%%%%%%%%%%%%%%%%%%%%%%%%%%%%%%%

% DEFINITION OF A COMPLEX LIMIT CYCLE

\begin{Def} \label{limitcycle}
A marked limit cycle is a marked complex cycle represented by an
isolated fixed point of the appropriate Poincar\'e map.
\end{Def}

%%%%%%%%%%%%%%%%%%%%%%%%%%%%%%%%%%%%%%%%%%%%%%%%%%%%%%%%%%%%%%%%%

% THE CLASSICAL ONE-FOLD CYCLES

The case when $m=1$ has been extensively studied. In fact, the
real cycles of a planar polynomial vector field of the form
(\ref{basic}) extend to $1-$fold cycles of its complexificaion.
The aforementioned infinitesimal Hilbert's 16th problem
\cite{BNY}, \cite{I02} treats exactly the special case $m=1.$ The
following classical result, known as Pontryagin's criterium
\cite{Po} can be stated in the following form.

\begin{thm}\label{thm-poncrit}
Let $\delta_{u}$ be an analytic family of simple closed loops on
the corresponding leaves $S_u$ from the integrable foliation
$\Foliation^0(\CC),$ and consider the analytic function $I(u)=
\int_{\delta_u} \omega.$ If there exists $u_0$ such that
$I(u_0)=0$ and $I'(u_0) \neq 0$ then there exists a continuous
family $\delta^{\e}$ of loops, each representing a 1-fold complex
limit cycle of $\Fol(\CC),$ such that $\delta^{\e} \to
\delta_{u_0}$ as $\e \to 0,$ always staying close to
$\delta_{u_0}.$
\end{thm}

%%%%%%%%%%%%%%%%%%%%%%%%%%%%%%%%%%%%%%%%%%%%%%%%%%%%%%%%%%%%%%%%%

% MAIN QUESTIONS ABOUT MULTI-FOLD CYCLES

In contrast to 1-fold cycles, little is known about multi-fold
ones. We are going to answer some questions related to the case
$m>1.$ During a series of informal discussions, Ilyashenko
proposed the following questions in the spirit of Petrovskii and
Landis' works \cite{PL55} and \cite{PL57}: \vspace{2mm}
%\begin{enumerate}

\noindent (1){\em Are there examples of polynomial families of
form (1) with Poncar\'e maps that have isolated periodic orbits of
arbitrary period $m>1$?} \vspace{1.5mm}

\noindent (2){\em In the case $m>1,$ what may happen to a
$\delta_0, m$-fold
 limit cycle when $\e$ approaches $0$?} \vspace{1.5mm}

\noindent (3) {\em Does a multi-fold limit cycle settle on a leaf
of $\Foliation^0(\CC)$ as $\e \to 0?$} \vspace{2mm}

%\end{enumerate}

The current article is an attempt to give answers to some of the
questions posed above. In order to do this, we construct a
Poincar\'e map on a large cross-section of the foliation
$\Fol(\CC),$ which we call \emph{a non-local Poincar\'e map}. We
show that a certain type of cycles of $\Fol(\CC)$ that generate a
periodic orbit of the Poincar\'e map, have representatives that
determine specific free homotopy classes of loops in an open
fibred subdomain of $\CC.$ The topology and fiber structure of
this subdomain is determined by $\Foliation^0(\CC).$ With the help
of the construction of the non-local Poincar\'e map we see that
the behavior of a multi-fold limit cycle is quite different from
the behavior of a 1-fold limit cycle as $\e$ tends to zero. By
Theorem \ref{thm-poncrit}, the latter always stays close to some
cycle from $\Foliation^0(\CC)$ and converges to it as $\e$
converges to zero. In contrast to the behavior of a $1-$fold limit
cycle, a multi-fold one tends either to escape from a very large
domain in $\CC$ when $\e$ approaches 0 or to change the homotopy
type of its representatives inside the fibred subdomain in $\CC$.
This phenomenon is called \emph{rapid evolution of the multi-fold
limit cycle}. We also give an explicit example of a polynomial
foliation of the form (\ref{basic}) with multi-fold limit cycles.

So far, the third question from the list above stays unanswered.
The information we have on rapid evolution reveals an interesting
insight. If the answer to that question is positive, then before a
multi-fold limit cycle can reach an algebraic leaf as $\e \to 0$,
its representatives should change their topological properties
somewhere along the way. This means that there is a possibility
that the cycle settles on a critical leaf of $\Foliation^0(\CC)$
or goes through one or several critical leaves of
$\Foliation^0(\CC),$ settling on a regular leaf. Since the
foliations are polynomial, they extend to foliations on $\cptwo.$
Thus, another possibility is an interaction with the line at
infinity.

\section{Main Results} \label{Section Main Thms}

\subsection{Preliminaries} \label{preliminaries}

In this section we define several fibred subdomains of the complex
plain that will play an important role in our investigation.

%%%%%%%%%%%%%%%%%%%%%%%%%%%%%%%%%%%%%%%%%%%%%%%%%%%
% THE POLYNOMIAL

Let the polynomial $H : \CC \to \cc$ be of degree $n+1$ and have
the following two properties:
%\begin{itemize}

\noindent $\bullet$ it has $n^2$ non-degenerate critical points in
$\CC$ with $n^2$ different critical values $\Sigma=\{a_1, ...,
a_n^2\}$ in $\cc$ and

\noindent $\bullet$ the projective closures of its leaves
$S_u=H^{-1}(u)$ in $\cptwo$ are transverse to the line at
infinity.

%%%%%%%%%%%%%%%%%%%%%%%%%%%%%%%%%%%%%%%%%%%%%%%%%%%%%%%%%
% THE SUB-DOMAINS B AND E

We are ready to define the first subdomain of $\CC.$ We are going
to denote it by $E.$ Consider the punctured domain $B=\cc -
\Sigma$ and its preimage $E=H^{-1}(B).$ Clearly, $E$ is just $\CC$
with all critical leaves of $H$ removed. Choose $u_0 \in B$ and
denote $S_{u_0}=H^{-1}(0).$ Then the map $H : E \to B$ defines a
smooth locally trivial fibre bundle with fibres diffeomorphic to
$S_{u_0}$ \cite{AGV}, \cite{IY}. Denote by $\Fol(E)$ the
restriction of the foliation $\Fol(\CC)$ on $E$. In other words,
the leaves of $\Fol(E)$ are the intersections of the leaves of
$\Fol(\CC)$ with $E.$ For simplicity, we are going to drop the
notation for $E$ in $\Fol(E)$ and just write $\Fol$ instead of
$\Fol(E)$. Thus $\Fol=\Fol(E)$. When $\e=0,$ the restricted
foliation $\Foliation^0=\Foliation^0(E)$ consists of all leaves
from $\Foliation^0(\CC)$ with the exception of the critical ones.

%%%%%%%%%%%%%%%%%%%%%%%%%%%%%%%%%%%%%%%%%%%%%%%%%%%%%%%%%%%
% TOPOLOGY OF THE FIBER BUNDLE

Before we go on with the construction of the other subdomains, we
will need some facts concerning the topology of the fiber bundle
$H : E \to B.$ For each critical value $a_j \in \Sigma, j=1 ...
n^2$ consider a simple smooth path from $u_0$ to a small circle
around $a_j,$ so that the union of the path and the circle
provides us with a counter clockwise oriented loop $\gamma_j$
around $a_j$ based at $u_0.$ Also, suppose that for $i \neq j,
\gamma_i \cap \gamma_j=\{u_0\}.$ Then the homotopy classes of the
loops $\{\gamma_1, ..., \gamma_{n^2}\}$ define generators of the
fundamental group $\pi_1(B, u_0)$. For $u \in \gamma_j$ consider
the fiber $S_u.$ Then if the parameter $u$ starts form $u_0$ and
moves along the loop $\gamma_j$ until it comes back to $u_0$ then
the corresponding fibers $S_u$ will also make one turn around the
critical point $a_j$ starting and ending up at $S_{u_0}.$
According to Picard-Lefschetz's theory \cite{AGV} this procedure
gives rise to an isotopy class of maps (an element of the mapping
class group of $S_{u_0}$) with a representative
$\tilde{D}_{\gamma_j} : S_{u_0} \to S_{u_0}$ which is a Dehn twist
around a simple closed geodesic we denote by $\delta_j$ for
$j=1,..., n^2.$ Moreover, the Dehn twist can be chosen so that the
closed cylinder $supp(\tilde{D}_{\gamma_j}) \subset S_{u_0},$ on
which $\tilde{D}_{\gamma_j}$ acts non trivially, is very thin with
respect to the Poincar\'e metric on the fiber $S_{u_0}$ and
$supp(\tilde{D}_{\gamma_i}) \cap
supp(\tilde{D}_{\gamma_j})=\varnothing$ whenever $\delta_i \cap
\delta_j=\varnothing.$ Then on the closure of the complement
$S_{u_0} - supp(\tilde{D}_{\gamma_j})$ the map
$\tilde{D}_{\gamma_j}$ acts like the identity map. The homotopy
classes represented by the loops $\{\delta_1,...,\delta_{n^2}\}$
give rise to a system of vanishing cycles on $S_{u_0},$ which can
serve as a basis of the first homology group on $S_{u_0}$
\cite{AGV}, \cite{I69}. Also, as a sphere with $n^2+1$ points
removed, $B$ has the structure of a Riemann surface with a
hyperbolic metric. For each cusp $a_j \in \Sigma$ let us choose a
cut $l_j$ connecting $\alpha_j$ to $\infty$ so that no two such
cuts intersect. For simplicity, we may think that each cut $l_j$
is geodesic and that $u_0$ is chosen so that it does not lie on
any of the cuts. Later, in Section \ref{global_map_construction}
we are going to find one possible way for those cuts to be chosen.

%%%%%%%%%%%%%%%%%%%%%%%%%%%%%%%%%%%%%%%%%%%%%%%%%%%%%%%%%%
% DOMAINS WITH CUTS Bdelta AND Edelta

Now we are ready to define the subdomain $\Edelta \subset E$. Fix
a point $p_0 \in S_{u_0}$ and some primitive element $\Delta_0$ of
the fundamental group $\pi_1(S_{u_0}, p_0).$ Choose a
representative $\delta_0 \subset S_{u_0}$ of $\Delta_0$ such that
$\delta_0 \cap supp(\tilde{D}_{\gamma_j}) = \varnothing$ if the
geometric intersection index $\delta_0 \cdot \delta_j = 0.$ Define
$J(\delta_0)=\{j=1, ...,n^2\, |\, \delta_0 \cdot \delta_j \neq
0\}$ to be the set of those indices for which the geometric
intersection index of the corresponding vanishing cycle and
$\delta_0$ is non zero and consider the domain $\Bdelta = B -
(\sqcup_{j \in J(\delta_0)}\, l_j) \subset \cc$ and $\Edelta =
H^{-1}(\Bdelta) \subset \CC.$

%%%%%%%%%%%%%%%%%%%%%%%%%%%%%%%%%%%%%%%%%%%%%%%%%%%%%%%%%%
% THE FOUR NESTED DOMAINS

Finally, we construct the rest of the domains. For a small number
$\tilde{\rho}>0,$ let us consider small disjoined closed discs
$B_1(\tilde{\rho}), ..., B_{n^2}(\tilde{\rho})$ of radius
$\tilde{\rho}$ in $\cc$ around the points $\alpha_1, ...,
\alpha_{n^2}$ respectively and not containing the point $u_0.$ Let
$B_{\infty}(\tilde{\rho})$ be a very large disc centered at the
origin and of radius $1/\tilde{\rho}$ so that it contains all of
the small ones and the point $u_0$. Then one can define the
domains $\Cdelta(\tilde{\rho}) = \Bdelta -
\bigl(B_{\infty}(\tilde{\rho}) \sqcup \bigl(\sqcup_{j=1}^{n^2}
B_j(\tilde{\rho})\bigr)\bigr)$ and $A(\tilde{\rho}) = B -
\bigl(B_{\infty}(\tilde{\rho}) \sqcup \bigl(\sqcup_{j=1}^{n^2}
B_j(\tilde{\rho})\bigr)\bigr).$ Fix four small positive numbers
$\rho_0, \rho_1, \rho^{\prime}_0$ and $\rho^{\prime}_1$,
satisfying the inequalities $\rho_0 > \rho_1 > \rho^{\prime}_0 >
\rho^{\prime}_1
> 0.$ Denote by $\Cdelta$ and $\Cdeltaprime$ the domains
$\Cdelta(\rho_0)$ and $\Cdelta(\rho_0^{\prime})$, respectively.
Also, denote by $A$ and $\Aprime$ the domains $A(\rho_1)$ and
$A(\rho_1^{\prime})$, respectively. Now, consider the preimages
$\Ecdelta = H^{-1}(\Cdelta)$ and $\EAprime = H^{-1}(\Aprime).$

%%%%%%%%%%%%%%%%%%%%%%%%%%%%%%%%%%%%%%%%%%%%%%%%%%%%%%%%%%%%%%%%%
%%%%%%%%%%%%%%%%%%%%%%%%%%%%%%%%%%%%%%%%%%%%%%%%%%%%%%%%%%%%%%%%%
% MAIN RESULTS

\subsection{Main Theorems and Statements}

Before stating the main results of this work, we are going to give
another definition for a multi-fold cycle. It is of a more
topological nature and, as point 4 from Theorem \ref{thm-pmap}
shows, in certain situations Definition \ref{mfold} and the new
definition coincide.

\begin{Def} \label{Definition_delta_m_fold_vertical}

A loop contained in $\Edelta$ is called $\delta_0, m$-fold
vertical provided that it is free homotopic to $\delta_0^m$ inside
$\Edelta$. A marked complex cycle of $\Fol$ is called $\delta_0,
m$-fold vertical provided that it has a $\delta_0, m$-fold
vertical representative contained in $\Edelta.$

\end{Def}

The justification for this definition stems from the proposition
that follows.

\begin{prop} \label{TopologyOfCycles} Let $\Fol$ have a marked
complex cycle $(\Delta,q)$ with a $\delta_0,m-$fold vertical
representative $\delta$ contained in $\Edelta.$ \vspace{1.5mm}

\noindent {\em 1.} If $\delta$ is free homotopic inside $\Edelta$
to another loop $\delta_0^{'} \subset S_{u_0},$ then
$\delta_0^{'}$ is free homotopic to $\delta_0^m$ on the fiber
$S_{u_0}.$ \vspace{1.5mm}

\noindent {\em 2.} If $\delta^{\prime}$ is another representative
of $(\Delta,q)$ contained in $\Edelta,$ then $\delta^{\prime}$ is
$\delta_0,m-$fold vertical.
\end{prop}

As we can see, a representative of a marked complex cycle can
belong to only one free homotopy class in $\Edelta.$ Moreover, any
other representative contained in $\Edelta$ belongs to the same
class.

The first main result of this paper is the construction of a large
cross-section of the foliations from the family (\ref{basic}) and
a Poincar\'e map defined on it. The result also shows that there
is a connection between the periodic orbits of the Poincr\'e map
and some topological properties of the corresponding multi-fold
cycles inside the fibered domain $\Edelta.$

%%%%%%%%%%%%%%%%%%%%%%%%%%%%%%%%%%%%%%%%%%%%%%%%%%%%%%%%%%%%%%%%
% TOPOLOGICAL CONSTRUCTIONS OF THE POINCAR\'E MAP AND TOPOLOGICAL
% PROPERTIES OF MULTI-FOLD CYCLES

\begin{thm}\label{thm-pmap}
There exists a surface $B_{p_0},$ embedded
in $E,$ diffeomorphic to $B$ and passing through $p_0,$ %with a subsurface $C_{p_0} \subset
%B_{p_0}$ diffeomorphic to $\Cdelta \subset B,$
such that $B_{p_0}$ intersects transversely each noncritical leaf
of $\Foliation^0$ at exactly one point. Moreover, for a small
enough $r>0,$ if $\e$ is contained in a disc of radius $r$ then
the following statements are true:

\vspace{1.5mm}

\noindent {\em 1.} The leaves of $\Fol$ are transverse to
$\Apzeroprime \subset B_{p_0},$ where $H(\Apzeroprime) = \Aprime.$
\vspace{1.5mm}

\noindent {\em 2.} Let $\Cpzeroprime \subset B_{p_0}$ be such that
$H(\Cpzeroprime) = \Cdeltaprime.$ Then there exists a Poincar\'e
map $\Pmap : \Cpzeroprime \to \Apzeroprime$ associated with the
foliation $\Fol$ and a complex structure on $B_{p_0}$ so that
$P_{\delta_0, \e}$ is holomorphic.\vspace{1.5mm}

\noindent {\em 3.} If $\Pmap$ has a periodic orbit of period $m$
in $\Cpzeroprime$ then the foliation $\Fol$ has a marked complex
cycle $(\Delta_{\e}, q_{\e})$ with a base point $q_{\e}$ belonging
to $\Cpzeroprime.$ Moreover, the cycle has a representative
$\delta_{\e}$ contained in $\EAprime$ and passing through the
points of the $m-$periodic orbit.\vspace{1.5mm}

\noindent {\em 4.}  If $\delta^{\prime}_{\e}$ is an arbitrary
representative of the marked complex cycle $(\Delta_{\e}, q_{\e})$
from point 3, then $\delta^{\prime}_{\e}$ is contained in
$\Edelta$ and is $\delta_0,m-$fold vertical if and only if its
image $H(\delta^{\prime}_{\e})$ is contained in $\Bdelta$ and is
free homotopic to a point inside $\Bdelta.$ %\vspace{1.5mm}

\end{thm}

%In the context of the last two claims form Theorem \ref{thm-pmap},
%the following terminology makes sense:

% Picture of the theorem 10

%%%%%%%%%%%%%%%%%%%%%%%%%%%%%%%%%%%%%%%%%%%%%%%%%%%%%%%%%%%%%%%
% DEFINITION OF VERTICAL MULTI-FOLD CYCLES

What we gain with this theorem is that for a small enough $\e$ we
are able to construct a Poincar\'e transformation along $\delta_0$
defined on a very large domain. In this way we can encode a lot of
information about a big portion of the perturbed foliation $\Fol.$
In particular, it allows us to keep track of the behavior of
\emph{continuous families} of $\delta_0, m-$fold limit cycles with
respect to the parameter $\e.$ In addition, Theorem \ref{thm-pmap}
reveals a link between the dynamical notion of a multi-fold cycle,
as given by Definition \ref{mfold} and the topological point of
view introduced in Definition
\ref{Definition_delta_m_fold_vertical}. Thus, there exists a
strong connection between the dynamical properties of marked
complex cycles, in terms of periodic orbits of the corresponding
Poincar\'e map, and the topological properties of these cycles, in
terms of free homotopy classes.

%%%%%%%%%%%%%%%%%%%%%%%%%%%%%%%%%%%%%%%%%%%%%%%%%%%%%%%%%%%%%%%%%
% CONTINUOUS FAMILY OF COMPLEX LIMIT CYCLES
%\paragraph{Rapid Evolution of Marked Limit Cycles.}
Next, we explain the notion of a \emph{continuous family} of
$\delta_0, m-$fold limit cycles with respect to a parameter $\e.$
\begin{Def}
A family $\{(\Deltae, q_{\e})\}_{\e}$ of limit cycles of $\Fol$ is
called continuous with respect to $\e,$ relative to an embedded in
$E$ surface $L,$ if there exists a continuous family of
representing loops from $\Deltae,$ so that the base point $q_{\e}$
varies continuously on $L$.
\end{Def}
The next main result shows that for $m>1,$ a continuous family of
$m-$fold limit cycles tends to escape from a very large domain in
$\CC$, namely $\Ecdelta.$ We refer to this phenomenon as
\emph{rapid evolution} of the multi-fold family. This behavior is
completely different from the behavior of a $1-$fold family.
According to Theorem 1, the latter always stays in a neighborhood
of an algebraic leaf of $\Foliation^0$ as $\e$ approaches 0.

%%%%%%%%%%%%%%%%%%%%%%%%%%%%%%%%%%%%%%%%%%%%%%%%%%%%%%%%%%%%%%%%
% RAPID EVOLUTION OF VERTICAL MULTI-FOLD LIMIT CYCLES

Fix a positive integer $m>1$ and let $D_r(0)=\{\e \in \cc
\,\,:\,\,|\e| \leq r\}$ for $r>0.$ %so that it satisfies Theorem
%\ref{thm-pmap} and in addition is small enough so that for any
%$k=0,..., m+1$ we have that $\Pdelta^k(\Cpzero) \subset \Apzero$
%and $\Pdelta^k(\Cpzeroprime) \subset \Apzeroprime.$ Denote by
%$D_r(0)$ the disc of radius $r$ and centered at $0$ in parameter space.
We claim that as long as $r>0$ is chosen small enough, rapid
evolution of marked complex cycles occurs in the following form:

\begin{thm}\label{rapid-evolution}
%Let $r>0$ be chosen according to Theorem \ref{thm-pmap} and let
%$m>1.$
Assume that for some $\e_0 \in D_r(0)$ the foliation
$\Foliation^{\e_0}$ has a $\delta_0,m$-fold vertical limit cycle
which corresponds to an $m$-periodic orbit of $P_{\delta_0, \e_0}$
on the cross-section $C_{p_0}^{\prime}.$ Also, assume that the
cycle has a $\delta_0, m-$fold vertical representative contained
in $\Ecdelta$. Then, for any curve $\eta$ connecting $\e_0$ to $0$
and embedded in $D_r(0)$, there exists a relatively open subset
$\sigma$ of $\eta,$ such that the cycle extends on $\sigma$ to a
continuous family $\{(\Delta_{\e}, q_{\e})\}_{\e \in \sigma}$ of
marked cycles of $\Fol$. Moreover, as $\e$ moves along $\sigma$ in
the direction of $0,$ it reaches a value $\estar \in \sigma$ such
that for any $\e \in \sigma$ past $\estar$ no $\delta_0, m-$fold
vertical representative of $(\Delta_{\e}, q_{\e})$ will be
contained in $\Ecdelta$ anymore. %and at the same time is free homotopic to
%$\delta_0^m$ inside $\Edelta.$
\end{thm}

To summarize the conclusions of Theorem \ref{rapid-evolution}, a
limit $\delta_0,m$-fold vertical cycle of the perturbed foliation,
represented by a periodic orbit of the corresponding Poincar\'e
map, gives rise to a continuous family defined on $\sigma$.
Eventually, as $\e$ goes in the direction of $0$ on $\sigma,$ all
representatives of the cycles from that family not only leave the
domain $\Ecdelta$ but they do not come back to it as multi-fold
vertical cycles of the same topological type. If they do come
back, their topological characteristics $\delta_0$ or $m$ are
changed.

%%%%%%%%%%%%%%%%%%%%%%%%%%%%%%%%%%%%%%%%%%%%%%%%%%%%%%%%%%%%%%5
% REMARK ABOUT RELATION BETWEEN PERIODIC ORBITS AND TOPOLOGICAL TYPE OF THE CYCLES

Before we continue with the exposition, we are going to make a
small comment. Denote by $\delta^{\prime}$ the representative of
the $\delta_0,m$-fold vertical cycle from Theorem
\ref{rapid-evolution} contained in the domain $\Ecdelta$ when
$\e=\e_0.$ Notice that as soon as its image $H(\delta^{\prime})$
is null-homotopic in $\Bdelta,$ the loop $\delta^{\prime}$ is
forced by point 4 from Theorem \ref{thm-pmap} to be free homotopic
inside $\Edelta$ to $\delta_0^m$ and cannot belong to any other
free homotopy class in $\Edelta$. Therefore, the fact that
$\Pmapezero$ is the Poincar\'e map with respect to $\delta_0$, is
directly related to the fact that $\delta^{\prime}$ is
$\delta_0,m-$vertical. Moreover, as Proposition
\ref{TopologyOfCycles} suggests, any other representatives of the
same marked cycle, contained in $\Edelta,$ will also be
$\delta_0,m-$fold vertical.

%%%%%%%%%%%%%%%%%%%%%%%%%%%%%%%%%%%%%%%%%%%%%%%%%%%%%%%%%%%%%%5
% PROOF OUTLINE OF FIRST MAIN RESULT

We are going to give short outlines of the proofs of the above two
results. To verify the claims of Theorem \ref{thm-pmap}, one can
use the pull back of the bundle $E$ over the universal covering
disc of the surface $B.$ In this way, a covering bundle with an
action of a deck group is obtained, and we can smoothly trivialize
that bundle (notice the disc is contractible) so that the group
will map both vertical and horizontal fibers to vertical and
horizontal fibers, respectively. In fact, the group preserves the
horizontal disc fibers passing through $S_{u_0} - \cup_{j=1}^{n^2}
supp(\tilde{D}_{\gamma_j})$ because on the vertical fibers it is
generated by the Dehn twists $\{\tilde{D}_{\gamma_j} \, : \,
j=1...n^2\},$ which act trivially outside
$supp(\tilde{D}_{\gamma_j}).$ In particular, if we take the
horizontal disc passing through $p_0$ and project it to $E,$ we
will obtain the desired cross-section $B_{p_0}.$ If we pull back
the foliation in the trivial bundle then we obtain a foliation
invariant with respect to the action of the deck group. The direct
product structure on the trivial covering bundle allows us to lift
$\delta_0$ on the leaves of the pulled back foliation so that we
get a Poincar\'e map $\hat{P}_{\delta_0, \e}$ on the disc. The
invariance of the foliation implies the relation $\gamma \circ
\hat{P}_{\delta_0, \e} =
\hat{P}_{\tilde{D}^{-1}_{\gamma}(\delta_0), \e} \circ \gamma$ for
all $\gamma \in \pi_1(B, u_0).$ But for $\delta_j \cdot \delta_0 =
0$ we have $\gamma_j \circ \hat{P}_{\delta_0, \e} =
\hat{P}_{\delta_0, \e} \circ \gamma_j$ because
$\tilde{D}_{\gamma_j}(\delta_0)=\delta_0.$ Projecting everything
back to $E,$ we get the desired cross-sections and Poincar\'e map.
By construction the map branches over the cuts of $\Bdelta.$ The
complex structure on $\Apzeroprime$ is defined as the transverse
structure to the leaves of $\Fol$ and extended by $0$ on $B_{p_0}
- \Apzeroprime.$ The remaining claims follow from the
constructions above.

%%%%%%%%%%%%%%%%%%%%%%%%%%%%%%%%%%%%%%%%%%%%%%%%%%%%%%%%%%%%%%
% PROOF OUTLINE OF RAPID EVOLUTION RESULT

When proving Theorem \ref{rapid-evolution}, one can use Theorem
\ref{thm-pmap} in order to represent the family of limit cycles as
an analytic family of $m-$periodic orbits of the corresponding
Poincar\'e map inside the cross-sections $\Cpzero.$ Then one can
apply a version of the known property that for $m>1$, an analytic
family of $m$-periodic orbits of a holomorphic map close to
identity, tends to escape a domain inside the map's definition. In
our case the domain happens to be $\Cpzero.$ Therefore, when the
points from the periodic orbit leave $\Cpzero$ they also leave
$\Ecdelta.$ Because all representatives pass through the base
point, and because the base point, which is a point from the
periodic orbit, happens to be outside $\Ecdelta,$ no
representative is entirely contained in $\Ecdelta.$ In the case
when the periodic orbit goes %entirely
through a cut, it turns into a periodic orbit of another branch of
the Poincar\'e map, obtained as a lift of a loop that can be sent
to the original $\delta_0$ by a Dehn twist. This implies that the
new cycle will not have representatives free homotopic to
$\delta_0^m$ inside $\Edelta$ anymore.

%%%%%%%%%%%%%%%%%%%%%%%%%%%%%%%%%%%%%%%%%%%%%%%%%%%%%%%%%%%%%%%
% THE PROBLEM OF EXISTENCE OF MULTI-FOLD LIMIT CYCLES

%\paragraph{Perturbed Foliations with Multi-Fold Limit Cycles.}
An important problem in the study of multi-fold limit cycles is
the existence of the latter in families of polynomial foliations
of the form (1). Heuristically, we can follow the following steps.
Using Theorem \ref{thm-poncrit}, we can find a family of
$\delta_0, 1$-fold cycles which gives a family of isolated fixed
points for the corresponding Poincar\'e map $P_{\e} = P_{\delta_0,
\e}.$ For infinitely many values of $\e$ in any neighborhood of 0,
the derivative of $P_{\e}$ evaluated at the fixed point will be an
$m$-th root of unity. Thus, for such $\e$ a local continuous
family of $m$-periodic isolated orbits will bifurcate from the
fixed point. This will happen as long as the resonant terms of the
normal form of the map do not vanish, i.e. the map is not
analytically equivalent to a rotation. Since having nonzero
resonant terms is a very generic property of resonant maps, we can
expect that the Poincar\'e transformations for most foliations of
the form (1) will have a lot of isolated periodic orbits and thus,
the foliations themselves will have many multi-fold limit cycles.
The only obstacle in this strategy is the verification that some
of the resonant term coefficients of the map's normal form are
nonzero. This is hard to establish since the connection between
the polynomial foliation and its Poincar\'e transformation is
implicit and indirect.

%%%%%%%%%%%%%%%%%%%%%%%%%%%%%%%%%%%%%%%%%%%%%%%%%%%%%%%%%%%%%
% EXAMPLE OF A FAMILY WITH MULTI-FOLD CYCLES

Modifying the strategy above, we give an example of a polynomial
foliation with limit multi-fold vertical cycles. Let $H$ be the
following polynomial with leaves transverse to infinity:
\begin{displaymath}
H = x^2 + y^2.
\end{displaymath}
Choose polynomial forms $\omega_1$ and $\omega_2$ as follows:
\begin{align*}
\omega_1 = (H-1)(y dx - x dy) \,\,\,\, \text{and} \,\,\,\,
\omega_2 = y \, dH.
\end{align*}
Consider the two parameter family
\begin{equation} \label{twoparameter}
\ker \Bigl(dH + \e(\omega_1 + a\omega_2)\Bigl),
\end{equation}
where $\e$ and $a$ are the parameters. Consider the leaf
$$S_1=\{(x,y) \in \CC \, | \, x^2 + y^2 =1\}$$ tangent to the
integrable line field $\ker(dH)$. Fix the loop $\delta_0 = S_1
\cap \RR$. In this setting, the following result holds:

\begin{thm} \label{example}
For any $m \in \Naturals$ large enough there exists a complex
parameter $\e_m$ near $\frac{1}{m}$ and a parameter $a_m$ such
that for all $\e$ in a neighborhood of $\e_m,$ the polynomial
foliation (\ref{twoparameter}) has a limit $\delta_0,m-$fold
vertical cycle. The cycle satisfies the properties of Theorem
\ref{thm-pmap} and is subject to rapid evolution, as explained in
Theorem \ref{rapid-evolution}.
\end{thm}

\section{Marked Cycles in the Fibred Domain}

\subsection{Topology of the Fiber Bundle}\label{TopologyFiberBundle}

First, we will try to understand the topology of the bundle $H : E
\to B$ induced by the integrable foliation $\Foliation^0.$ The
idea is to "unfold" $E$ into something simple, a direct product in
our case, keeping the "folding pattern" into a group of deck
transformations.

Let $\Disc$ be the open unit disc in $\cc.$ Consider the universal
covering map $\pi : \Disc \to B.$ Denote its group of deck
transformations by $\Gamma.$ Then, $\Gamma$ is isomorphic to the
fundamental group of $B$. Since the disc $\Disc$ is a conformal
model of the hyperbolic plane, $\Gamma$ is a discreet group of
isometries acting properly discontinuously. Let $z_0 \in \Disc$ be
a point such that $\pi(z_0)=u_0.$ Each loop $\gamma_j \in
\pi_1(B,u_0)$ we have chosen in Section \ref{preliminaries} can be
lifted to a path on $\Disc$ starting from point $z_0$. Denote by
$z^{(j)}$ the second end of this path. Abusing notation, for each
$j=1...n^2$ consider $\gamma_j \in \Gamma$ to be the parabolic
isometry of $\Disc$ corresponding to the loop $\gamma_j \in
\pi_1(B,u_0)$ that sends $z_0$ to $z^{(j)}=\gamma_j(z_0).$ Then
$\Gamma = \langle \gamma_1, ..., \gamma_{n^2} \rangle$ is a free
group generated by $n^2$ transformations. Let $\hat{a}_j$ be the
fixed point of the parabolic isometry $\gamma_j$ on the boundary
$\partial \Disc$ for all $j=1,...,n^2.$ We can think of
$\hat{a}_j$ as the lift of $a_j \in \Sigma$ on the ideal boundary
$\partial \Disc$ of the hyperbolic plane $\Disc.$ Assume that the
subscripts in the notation of the critical values are chosen so
that the loop $\gamma_{n^2}...\gamma_2 \gamma_1$ on $B$ is
homotopic to a simple loop around the cusp $\infty$ of $B.$ Thus,
the corresponding isometry $\gamma_{n^2} \circ ...\circ \gamma_2
\circ \gamma_1 \in \Gamma$ is also parabolic with a fixed point we
denote by $\infty_1 \in \partial \Disc$ which can be thought of as
a lift of the infinity point of $\cc \cup \{ \infty \}$ on the
ideal boundary $\partial \Disc.$ Similarly, for any $j=2...n^2$
the isometry $\gamma_{j-1} \circ...\circ \gamma_1 \circ
\gamma_{n^2} \circ...\circ \gamma_j \in \Gamma$ is parabolic with
a fixed point $\infty_j \in \partial \Disc.$ The ideal points
$\hat{a}_1, \infty_1, \hat{a}_2, \infty_2,..., \hat{a}_{n^2}$ and
$\infty_{n^2}$ are arranged in a cyclic order along the boundary
$\partial \Disc$. The geodesic convex hull of those $2n^2$ points
with respect to the Poincar\'e metric on $\Disc$ is %The choice
%of these particular generators gives rise to
a closed (in the topology of $\Disc$) ideal $2n^2-$gon $Q$ with
geodesic edges, which is a fundamental domain for the deck group
$\Gamma.$

From now on, we are going to use the shorter notation $S$ for the
fixed fiber $S_{u_0}.$ Also, whenever we have a cartesian product
$M_1 \times M_2$ of two sets, by $pr_{M_i}$ we are going to denote
the projection $pr_{M_i} : M_1 \times M_2 \to M_i$ where
$pr_{M_i}(m_1,m_2)=m_i$ for $i=1,2.$

\begin{thm} \label{bundle_topology}
There is a smooth covering map $\Pi : \Disc \times S \to E$ with
the following properties: \vspace{1.5mm}

\noindent {\em 1.} If $pr_{\Disc} : \Disc \times S \to \Disc$ is
the projection $(z,p) \mapsto z$ then $H \circ \Pi =\pi \circ
pr_{\Disc}.$ \vspace{1.5mm}

\noindent {\em 2.} The deck group of $\Pi : \Disc \times S \to E$
is $$\hat{\Gamma} = \langle \,\, (z,p) \mapsto (\gamma_j(z),
D_{\gamma_j}(p)) \,\,|\,\, j=1...n^2 \,\, \rangle,$$ where
$\gamma_j \in \Gamma$ are the earlier described generators of
$\Gamma$ and the maps $D_{\gamma_j}=\tilde{D}_{\gamma_j}^{-1}$ are
Dehn twists along the vanishing cycles $\delta_j$ on the surface
$S.$ Thus the factor bundle $(\Disc \times S) / \hat{\Gamma}$ is
diffeomorphically isomorphic to the bundle $E.$ %\vspace{2mm}
\end{thm}

The essence of this theorem is that not only we can unfold the
bundle $H : E \to B$ into a trivial covering bundle $pr_{\Disc} :
\Disc \times S \to \Disc$ but we can do so by making sure the deck
group $\Gammahat$ acts in a very special manner. It is natural to
expect that any element of the group takes vertical fibers $\{z\}
\times S$ to vertical fibers. What is important is that it also
sends horizontal fibers $\Disc \times \{p\}$ to horizontal fibers.

%\paragraph{\em Proof:}
\begin{proof}
Consider the pullback of the bundle $H : E \to B$ over the disc
$\Disc$ under the covering map $\pi.$ To carry out this
construction, first define the total space $\pi^{*}E=\{(z,q) \in
\Disc \times E \,\, : \,\, \pi(z)=H(q)\}.$ Then, the restricted
projection $\kappa=(pr_{\Disc})_{|_{\pi^{*}E}} : \pi^{*}E \to
\Disc$ gives us
the desired pullback bundle. %Here, $pr_{\Disc}(z,q)=z$ is the
%standard projection of $\Disc \times E$ onto $\Disc.$
Also, there is a map $\tilde{\Pi}^{'}=(pr_{E})_{|_{\pi^{*}E}} :
\pi^{*}E \to E$ %where $pr_E(z,q)=q.$
that satisfies the condition $H \circ \tilde{\Pi}^{'} = \kappa
\circ \pi$ and so it is a bundle map over the map $\pi.$ Together
with that, $\tilde{\Pi}^{'} : \pi^{*}E \to E$ is a covering map.

Because $\Disc$ is contractible, the pullback bundle $\kappa :
\pi^{*}E \to \Disc$ is trivializible, i.e. there is a smooth
bundle isomorphism $\varsigma : \Disc \times S \to \pi^{*}E$ so
that we have $\kappa \circ \varsigma = pr_{\Disc} \circ
id_{\Disc}$ where $id_{\Disc}$ is the identity map on $\Disc$.
Then, the composition $\tilde{\Pi}=\tilde{\Pi}^{'} \circ \varsigma
: \dtimess \to E$ satisfies the condition $H \circ \tilde{\Pi}=
\pi \circ pr_{\Disc}$ and thus it is a bundle map and
a covering map at the same time. %Denote by $z_0 \in Q \subset
%\Disc$ the unique point in the fundamental domain $Q$ that maps to $u_0$ under the map $\pi.$
Without loss of generality we can think that
$\tilde{\Pi}(z_0,p)=p,$ that is we identify the fiber $\{z_0\}
\times S$ with the surface $S.$

We are going to look at the deck group $\Gammatilde$ of the
covering map $\tilde{\Pi}.$ Let $\tilde\gamma \in \Gammatilde$ be
a deck transformation from that group. Then the diffeomorphism
$\tilde\gamma : \dtimess \to \dtimess$ is of the form
$\tilde{\gamma}(z,p)=(\gamma(z),\psi_{\gamma}(z,p))$ where $\gamma
\in \Gamma$ is a deck transformation for the covering map $\pi$
and $\psi : \dtimess \to S$ is a smooth map. If we factor
$\dtimess$ by the action of the deck group $\Gammatilde$ we obtain
the manifold $(\dtimess) / \Gammatilde$ which is isomorphic to $E$
as a fiber bundle over $B$. For any $(z,p) \in \dtimess$ consider
$\psi_{\gamma,z}(p)=\psi_{\gamma}(z,p).$ Then, $\psi_{\gamma,z} :
S \to S$ is a diffeomorphism on the standard fiber $S$ for any
fixed $z \in \Disc.$ If $\gamma_j$ is one of the generators of
$\Gamma$, as described before, then $\psi_{\gamma_j,z_0}$ is
isotopic to the Dehn twist
$D_{\gamma_j}=\tilde{D}_{\gamma_j}^{-1}.$ This follows from
Picard-Lefchetz's theory as discussed previously in section
\ref{preliminaries} and in \cite{AGV}.

By the properties of the ideal polygon $Q,$ for each $j=1,...,n^2$
there are two adjacent geodesic edges that have $\hat{\alpha}_j$
as a common ideal vertex. One of those two edges, we denote by
$e_j,$ is mapped by $\gamma_j$ to the other one, we denote by
$\gamma_j(e_j).$ Then, both $e_j$ and $\gamma_j(e_j)$ meet the
ideal boundary $\partial \Disc$ at $\hat{\alpha_j}.$ Now, for any
$j=1,...,n^2$ consider an open tubular neighborhood $I_j$ of $e_j$
in $\Disc$ thin enough so that two properties hold. First,
$\overline{I}_i \cap \overline{I}_j=\emptyset$ whenever $i \neq
j.$ Here, $\overline{I}_j$ is the closure of $I_j$ in the
hyperbolic plane $\Disc.$ Second, $\overline{I_j} \cap
\gamma_j(\overline{I}_j)=\emptyset,$ where $j=1,...,n^2.$ Notice,
that $\gamma_j(I_j)$ is a tubular neighborhood of $\gamma_j(e_j).$
Let $I = \sqcup_{j=1}^{n^2} I_j$ and $J=\sqcup_{j=1}^{n^2}
\gamma_j(I_j).$ Denote by $\tilde{Q}$ the union $Q \cup I \cup J.$
We can see that $\tilde{Q}$ is an open neighborhood of the
fundamental domain $Q.$ %Define a smooth map $\phi : I \to J$ so
%that $\phi_{|_{I_j}} = \gamma_j$ for each $j=1,...,n^2.$ Notice
%that the quotient $\tilde{Q} / \phi$ is diffeomorphic to the
%quotient $\Disc / \Gamma$ and therefore is diffeomorphic to the
%surface $B.$

Define the smooth gluing map $\phi_0 : I \times S \to J \times S$
to be $\phi_{0}(z,p)=(\gamma_j(z),\psi_{\gamma_j}(z,p))$ for any
$(z,p) \in I_j \times S,$ where $j=1,...,n^2.$ Since $\phi_0$
respects the bundle structure of $\dtimess$, the quotients
$(\tilde{Q} \times S)/\phi_0$ and $(\dtimess)/\Gammatilde$ are
smoothly isomorphic as fiber bundles over $B$ (for isotopies of
gluing maps, see for example \cite{Hr}.) Therefore, $(\tilde{Q}
\times S)/\phi_0$ and $E$ are smoothly isomorphic as bundles over
$B.$

Notice, that $I_j$ is diffeomorphic to a disc and so it
deformation retracts onto a point $z_j \in I_j$ for $j=1,...,n^2.$
For that reason, there exists a smooth deformation retraction
$r^{(j)} : I_j \times [0, \third] \to I_j$ so that
$r^{(j)}_0=id_{I_j}$ and $r^{(j)}_{1/3}\equiv z_j.$ Then, extend
$r^{(j)}_t$ smoothly for $t \in [0,\twothird]$ so that whenever $t
\in [\third,\twothird]$ we have $r^{(j)}_t(z)=z_j(t)$ for any $z
\in I_j$ where $z_j(t)$ is a smoothly parametrized geodesic
connecting $z_j$ to $z_0.$ Thus the smooth map $r^{(j)} : I_j
\times [0, \twothird] \to I_j$ is a homotopy connecting the
identity map on $I_j$ to the constant map $r^{(j)}_{2/3}(z)=z_0$
for $z \in I_j.$

Define the isotopy
\begin{align*}
 &\phi : I \times S \times [0,2/3] \to J \times S \\
 &\phi_t(z,p)=(\gamma_j(z),\psi_{\gamma_j}(r^{(j)}_t(z),p))
\end{align*}
for $(z,p) \in I_j \times S$ where $j=1,...,n^2.$ When $t=0$ we
have the earlier defined map $\phi_0.$ When $t=2/3$ we obtain the
map $\phi_{2/3}(z,p)=(\gamma_j(z),\psi_{\gamma_j}(z_0,p))$ for
$(z,p) \in I_j \times S.$ Notice that the second component of
$\phi_{2/3}$ does not depend on the variable $z$ but only on $p.$
As we mentioned earlier,
$\psi_{\gamma_j}(z_0,p)=\psi_{\gamma_j,z_0}(p)$ is isotopic to
$D_{\gamma_j}(p).$ Let
$\Psi^{j}_t(z,p)=\psi_{\gamma_j}(r^{(j)}_t(z),p)$ for $t \in
[0,\twothird]$ and $(z,p) \in I_j \times S$ where $j=1,...,n^2.$
Let $\Psi^{j}_t(z,p)$ for $t \in [\twothird, 1]$ be the isotopy on
the surface $S$ that connects the diffeomorphism
$\psi_{\gamma_j,z_0}(p)$ to the Dehn twist
$D_{\gamma_j}=\tilde{D}_{\gamma_j}^{-1}.$ Notice, that in the case
when $t \in [\twothird, 1]$ the presence of the variable $z$ in
the expression $\Psi^{j}_t(z,p)$ is superficial as the isotopy in
fact does not depend on $z$ but it takes place only on the surface
$S.$

Using the notation above, define the isotopy
\begin{align*}
 &\phi : I \times S
\times [0,1] \to J \times S \\
&\phi_t(z,p)=(\gamma_j(z),\Psi^{j}_t(z,p))
\end{align*}
for $(z,p) \in I_j \times S$ where $j=1,...,n^2.$ Thus, for
$j=1,...,n^3$ the maps
$\phi_0(z,p)=(\gamma_j(z),\psi_{\gamma_j}(z,p))$ and
$\phi_1(z,p)=(\gamma_j(z), D_{\gamma_j}(p))$ are isotopic for
$(z,p) \in I_j \times S$. Notice that $\phi_t$ respects the
vertical fibers $\{z\} \times S,$ that is the isotopy takes place
only with respect to the second coordinate, along the fiber $S,$
while the first coordinate is kept the same. Therefore,
$(\tilde{Q} \times S)/\phi_0$ and $(\tilde{Q} \times S)/\phi_1$
are smoothly isomorphic as fiber bundles over $B.$ As we already
saw, $(\tilde{Q} \times S)/\phi_0$ and $E$ %$(\Disc \timesS)/\Gammatilde$
are isomorphic as well. Hence, %$(\Disc \times S)/\Gammatilde$
$(\tilde{Q} \times S)/\phi_1$ and $E$ are isomorphic as bundles
over $B.$ Since by construction $(\tilde{Q} \times S)/\phi_1$ and
$(\Disc \times S)/\Gammahat$ are also isomorphic as bundles over
$B$, we can conclude that there exists a smooth bundle isomorphism
$\Phi : (\Disc \times S)/\Gammahat \to E.$ If $\upsilon : \dtimess
\to (\dtimess)/\Gammahat$ is the quotient map, then it is a bundle
map over the covering map $\pi.$ When we compose it with $\Phi$ we
obtain the desired bundle covering map $\Pi = \Phi \circ \upsilon
: \dtimess \to E$ satisfying the condition $H \circ \Pi=\pi \circ
pr_{\Disc}$ and having $\Gammahat$ as its group of deck
transformations. This completes the proof of the theorem.
\end{proof}

The results from Theorem \ref{bundle_topology} are a main tool in
the proofs of Theorem \ref{thm-pmap} and \ref{rapid-evolution}. As
it was mentioned already, a deck transformation $\hat{\gamma}(z,p)
= (\gamma(z), D_{\gamma}(p))$ from $\hat{\Gamma}$ maps not only
vertical fibers $\{z\} \times S$ to vertical fibers $\{\gamma(z)\}
\times S$ but also horizontal fibers $\Disc \times \{p\}$ to
horizontal fibers $\Disc \times \{D_{\gamma}(p)\}$. In particular,
since $D_{\gamma}$ acts on $S_{u_0} - (\cup_{j=1}^{n^2}
supp(D_{\gamma_j}))$ as the identity map, whenever $p\in S_{u_0} -
(\cup_{j=1}^{n^2} supp(D_{\gamma_j})),$ the horizontal disc $\Disc
\times \{p\}$ is invariant under the action of $\hat{\Gamma}.$
These facts lead us to the following conclusion.

\begin{cor} \label{transverse surface}
The projection $\Pi(\Disc \times \{p\}) = B_p$ is a smoothly
embedded surface in $E,$ diffeomorphic to $B.$ It intersects each
leaf from the integrable foliation $\Foliation^0$ transversely at
a single point.
\end{cor}

In particular, this corollary applies to the point $p_0.$ Thus, we
have obtained the global cross-section $B_{p_0}.$

\subsection{Properties of Multi-Fold Vertical
Cycles}\label{Section_Vertical_Cycles}

In this section, we give a proof of Proposition
\ref{TopologyOfCycles}. We start with some notations which will be
used at a later time.

Let $M$ be an arbitrary path-connected topological space with a
base point $x_0 \in M.$ Let $l$ be a loop on $M$ passing through
$x_0.$ Then, by $[l]_M$ we are going to denote the equivalence
class of all loops homotopic to $l$ in $M,$ relative to the base
point $x_0.$

Denote by $\Bdeltahat \subset \Disc$ the connected component of
$\pi^{-1}(\Bdelta)$ that contains the point $z_0.$ First, the
domain $\Bdeltahat$ is open. Second, the closure of $\cup_{\gamma
\in \Gamma} \gamma(\Bdeltahat)$ is equal to the whole disc
$\Disc.$ Third, for any two transformations $\gamma_1$ and
$\gamma_2$ from $\Gamma$, either $\gamma_1(\Bdeltahat) \cap
\gamma_2(\Bdeltahat) = \varnothing$ or $\gamma_1(\Bdeltahat) =
\gamma_2(\Bdeltahat).$

Since $\Bdeltahat$ is homeomorphic to a disc, there exists a
deformation retraction $\overline{R}_t : \Bdeltahat \to
\Bdeltahat$ of $\Bdeltahat$ onto $z_0,$ where $t \in [0,1].$ Then
$\overline{R}_0=id_{\Bdeltahat}, \overline{R}_1 \equiv z_0$ and
$\overline{R}_t(z_0)=z_0$ for all $t \in [0,1].$ Using
$\overline{R}_t$, we can define the continuous one-parameter
family of maps $R_t : \Bdeltahat \times S \to \Bdeltahat \times S$
by denoting $R_t(z,p)=(\overline{R}_t(z),p),$ where $t \in [0,1]$
and $(z,p) \in \Bdeltahat \times S.$  Notice, that $R_0 =
id_{(\hat{\Bdelta} \times S)}$ and $R_1(z,p)=(z_0,p).$ In
addition, $R_t(z_0,p)=(\overline{R}_t(z_0),p)=(z_0,p)$ for any
point $(z_0,p) \in \{z_0\} \times S$ and any $t \in [0,1]$. Then
$R_t$ is a deformation retraction of $\Bdeltahat \times S$ onto
$\{z_0\} \times S.$ For simplicity, let $R=R_1.$ So
$R(z,p)=(z_0,p)$ for any $(z,p) \in \Bdeltahat \times S$ and it
can be rewritten as $R(z,p)=(z_0,pr_S(z,p)).$

Analogously, we can define a deformation retraction $R^{\prime}_t$
of $\Disc \times S$ onto $\{z_0\} \times S.$ Again for simplicity,
we denote $R^{\prime}(z,p)=R^{\prime}_1(z,p)=(z_0,p)$ for any
point $(z,p)$ from $\Disc \times S.$ As in the case of $R$, we can
write $R^{\prime}(z,p)=(z_0,pr_S(z,p))$ \vspace{2mm}
\paragraph{\em Proof of Proposition \ref{TopologyOfCycles}.}
We start with point one from the proposition. By assumption, we
know that the foliation $\Fol$ has a marked cycle $(\Delta,q)$
with a representative $\delta$ contained in $\Edelta$ and free
homotopic to $\delta_0^m$ inside $\Edelta.$ Assume that besides
that, the representative $\delta$ is free homotopic inside
$\Edelta$ to another loop $\delta_0^{\prime},$ also lying on the
fibre $S.$ This implies that there exists a free homotopy
$\delta(t)$ inside $\Edelta,$ where $t \in [0,1],$ such that
$\delta(0)=\delta_0^{\prime}$ and $\delta(1)=\delta_0.$ The loop
$\delta_0^{\prime}$ lifts to the loop $\{z_0\} \times
\delta_0^{\prime}$ on the fiber $\{z_0\} \times S$ and so,
$\Pi(\{z_0\} \times \delta_0^{\prime}) = \delta_0^{\prime}.$ Then
$\delta(t)$ lifts to a homotpy $\hat{\delta}(t)$ for which
$\hat{\delta}(0)=\{z_0\} \times \delta_0^{\prime}.$ When $t=1$ the
loop $\hat{\delta}(1)$ belongs to the fiber $\{\gamma(z_0)\}
\times S$ and maps to $\delta_0=\Pi(\hat{\delta}(1)),$ where
$\gamma \in \Gamma.$ Since the homotopy $\delta(t)$ takes place
inside the domain $\Edelta,$ the lifted homotopy $\hat{\delta}(t)$
takes place in $\Bdeltahat \times S,$ so in fact $\gamma \in
\Gamma_0.$ Because $\hat{\delta}(1)$ lies on the fiber
$\{\gamma(z_0)\} \times S,$ it has the form $\hat{\delta}(1) =
\{\gamma(z_0)\} \times \delta_1,$ where $\delta_1$ is a loop on
the surface $S.$ Using this representation we compute
\begin{align*}
\Pi\Bigl(\{\gamma(z_0)\} \times \delta_1\Bigr)&=\Pi \circ
\hat{\gamma}^{-1}\Bigl(\{\gamma(z_0)\} \times \delta_1\Bigr)\\
&=\Pi\Bigl(\{\gamma^{-1} \circ \gamma(z_0)\} \times
D^{-1}_{\gamma}(\delta_1)\Bigr) \\
&=\Pi\Bigl(\{z_0\} \times D^{-1}_{\gamma}(\delta_1)\Bigr) \\
&= D^{-1}_{\gamma}(\delta_1)=\delta_0,
\end{align*}
that is $\delta_1=D_{\gamma}(\delta_0).$ Now, consider the
homotopy $pr_S(\hat{\delta}(t))$ which takes place only on the
surface $S$. Notice that $pr_S(\hat{\delta}(t))$ is continuous
with respect to $t \in [0,1].$  Moreover, for $t=0$ we have
$pr_S(\hat{\delta}(0))=\delta_0^{\prime}$ and for $t=1$ we have
$pr_S(\hat{\delta}(1))=\delta_1=D_{\gamma}(\delta_0).$ As we
already noticed, $D_{\gamma}(\delta_0)=\delta_0$ whenever $\gamma
\in \Gamma_0,$ hence $pr_S(\hat{\delta}(t))$ is the desired
homotopy on the surface $S$ between the two loops
$\delta_0^{\prime}$ and $\delta_0.$

Next, we prove the second part of the proposition. Since both
$\delta$ and $\delta^{\prime}$ are representatives from the same
marked cycle $(\Delta,q),$ there exists a homotopy $\delta(t)$ on
the leaf $\varphi^{\e}_q$ that keeps the base point $q$ fixed and
connects $\delta$ to $\delta^{\prime}$. Ignoring the leaf
$\varphi^{\e}_q,$ we have a homotopy $\delta(t)$ inside $E$ such
that $\delta(0)=\delta$ and $\delta(1)=\delta^{\prime}.$

Let $(\tilde{z},\tilde{p}) \in \Bdeltahat \times S$ be such that
$\Pi(\tilde{z},\tilde{p})=q.$ Since $\delta$ is $\delta_0,m-$fold
vertical, it lifts under the covering map $\Pi$ to a loop
$\hat{\delta}$ contained in $\Bdeltahat \times S.$ By the homotopy
lifting property of covering spaces \cite{H}, the homotopy
$\delta(t)$ inside $E$ lifts to a homotopy $\hat{\delta}(t)$
inside $\Disc \times S$, so that $\Pi(\hat{\delta}(t))=\delta(t).$
Thus, $\hat{\delta}(t)$ connects $\hat{\delta}$ to
$\hat{\delta}^{\prime}=\hat{\delta}(1),$ where
$\Pi(\hat{\delta}^{\prime})=\delta^{\prime}.$

Because of the assumption that $\delta^{\prime}$ is contained in
$\Edelta$, it follows that $\hat{\delta}^{\prime}$ is inside
$\gamma(\Bdeltahat) \times S$ for some $\gamma \in \Gamma.$ Then,
the base point $(\tilde{z},\tilde{p}),$ which lies on the loop
$\hat{\delta}^{\prime},$ is simultaneously in $\gamma(\Bdeltahat)
\times S$ and in $\Bdeltahat \times S.$ Therefore
$(\gamma(\Bdeltahat) \times S) \cap (\Bdeltahat \times S) \neq
\varnothing,$ which is possible only when $\gamma(\Bdeltahat) \cap
\Bdeltahat \neq \varnothing.$ But by construction,
$\gamma(\Bdeltahat) \cap \Bdeltahat \neq \varnothing$ if and only
if $\gamma(\Bdeltahat) = \Bdeltahat.$ It follows from here that
$\hat{\delta}^{\prime}$ is contained in $\Bdeltahat \times S.$

%The domain $\Bdeltahat \subset \Disc$ is homeomorphic to a disc.
As pointed out in the two paragraphs preceding the proof, the map
$R : \Bdeltahat \times S \to \{z_0\} \times S$ defined by the
expression $R(z,p)=(z_0,p)$ is a deformation retraction.
Similarly, $R^{\prime} : \Disc \times S \to \{z_0\} \times S$,
defined by the same rule $\hat{R}^{\prime}(z,p)=(z_0,p)$, is also
a deformation retraction. The induced homomorphisms on the
corresponding fundamental groups
\begin{align*}
& R_{*} : \pi_1(\Bdeltahat \times S, (\tilde{z},\tilde{p})) \to
\pi_1(\{z_0\} \times S,(z_0,\tilde{p})) \\
& R_{*}^{\prime} : \pi_1(\Disc \times S, (\tilde{z},\tilde{p}))
\to \pi_1(\{z_0\} \times S,(z_0,\tilde{p})),
\end{align*}
given by $R_{*}[\,l\,]_{(\Bdeltahat \times S)}=[R(l)]_{(\{z_0\}
\times S)}$ and $R_{*}^{\prime}[\,l^{\prime}\,]_{(\Disc \times
S)}=[R^{\prime}(l^{\prime})]_{(\{z_0\} \times S)}$ respectively,
are isomorphisms since they come from deformation retractions
\cite{H}. Here, $l$ and $l^{\prime}$ are arbitrary loops from
$\Bdeltahat \times S$ and $\Disc \times S$ respectively, passing
through $(\tilde{z},\tilde{p})$. Because of the fact that $R$ is
simply the restriction of $R^{\prime}$ onto $\Bdeltahat \times S$
and that both loops $\hat{\delta}$ and $\hat{\delta}(1)$ are
contained in $\Bdeltahat \times S,$ it follows that
\begin{align*}
R_{*}^{\prime}[\,\hat{\delta}\,]_{(\Disc \times S)} &=
[R^{\prime}(\hat{\delta})]_{(\{z_0\} \times
S)}=[R(\hat{\delta})]_{(\{z_0\} \times
S)}=R_{*}[\,\hat{\delta}\,]_{(\Bdeltahat \times S)}\\
R_{*}^{\prime}[\,\hat{\delta}^{\prime}\,]_{(\Disc \times S)}
&=[R^{\prime}(\hat{\delta}^{\prime})]_{(\{z_0\} \times
S)}=[R(\hat{\delta}^{\prime})]_{(\{z_0\} \times
S)}=R_{*}[\,\hat{\delta}^{\prime}\,]_{(\Bdeltahat \times S)}.
\end{align*}
Since $\hat{\delta}$ and $\hat{\delta}^{\prime}$ are homotopic
inside $\Disc \times S$ via $\hat{\delta}(t),$ we can see that
$[\,\hat{\delta}\,]_{(\Disc \times
S)}=[\,\hat{\delta}^{\prime}\,]_{(\Disc \times S)}.$ Therefore,
$R_{*}^{\prime}[\,\hat{\delta}\,]_{(\Disc \times
S)}=R_{*}^{\prime}[\,\hat{\delta}^{\prime}\,]_{(\Disc \times S)}.$
Combining all of those identities, we obtain
\begin{align*}
R_{*}[\,\hat{\delta}\,]_{(\Bdeltahat \times S)} &=
R_{*}^{\prime}[\,\hat{\delta}\,]_{(\Disc \times
S)}=R_{*}^{\prime}[\,\hat{\delta}^{\prime}\,]_{(\Disc \times
S)}=R_{*}[\,\hat{\delta}^{\prime}\,]_{(\Bdeltahat \times S)}.
\end{align*}
Since $R_{*}$ is a group isomorphism
\begin{align*}
R_{*}[\,\hat{\delta}\,]_{(\Bdeltahat \times
S)}=R_{*}[\,\hat{\delta}^{\prime}\,]_{(\Bdeltahat \times S)}\,
\,\, &\text{if and only if} \,\,\, [\,\hat{\delta}\,]_{(\Bdeltahat
\times S)}=[\,\hat{\delta}^{\prime}\,]_{(\Bdeltahat \times S)},
\end{align*}
which immediately implies that there exists a homotopy
$\hat{\delta}_t$ inside $\Bdeltahat \times S$ such that
$\hat{\delta}_0=\hat{\delta}$ and
$\hat{\delta}_1=\hat{\delta}^{\prime}.$ The projection of
$\hat{\delta}_t$ back to $E$ gives rise to a homotopy
$\delta_t=\Pi(\hat{\delta}_t)$ inside $\Edelta$ between the loops
$\delta^{\prime}$ and $\delta.$ By assumption, $\delta$ is free
homotopic to $\delta_0^m$ inside $\Edelta.$ Therefore,
$\delta^{\prime}$ is also free homotopic to $\delta_0^m$ inside
$\Edelta.$ $\square$

%%%%%%%%%%%%%%%%%%%%%%%%%%%%%%%%%%%%%%%%%%%%%%%%%%%%%%%%%%%%%%
%%%%%%%%%%%%%%%%%%%%%%%%%%%%%%%%%%%%%%%%%%%%%%%%%%%%%%%%%%%%%%%
% TOPOLOGICAL CONSTRUCTION OF THE POINCAR\'E MAP

\section{The Poincar\'e Map, Periodic Orbits, and Marked Cycles}

The goal of this section is to provide the proof of Theorem
\ref{thm-pmap}. It heavily relies on the results from the
preceding chapter and establishes the link between the topological
properties of the foliation and the dynamical properties of its
Poincar\'e transformation, constructed on a very large
cross-section.

\subsection{Construction of a Non-Local Poincar\'e
Map}\label{global_map_construction}

%%%%%%%%%%%%%%%%%%%%%%%%%%%%%%%%%%%%%%%%%%%%%%%%%%%%%%%%%%%%%%%5
% CUTS

As promised in Section \ref{preliminaries}, we begin with a
description of each cut $l_j$ that connects the cusp $a_j$ to
$\infty$ on $B,$ for $j \in J(\delta_0)$. Let $l_ j = \pi(e_j) =
\pi(\gamma_j(e_j)) \subset B$ be the image of the two adjacent
geodesic edges $e_j$ and $\gamma_j(e_j)$ of the ideal polygon $Q$
that meet the boundary of $\Disc$ at $\hat{a}_j$ (see Section
\ref{preliminaries}.)

%%%%%%%%%%%%%%%%%%%%%%%%%%%%%%%%%%%%%%%%%%%%%%%%%%%%%%%%%%%%%%%%%
%  LIFTED DOMAIN

Now, having in mind all the constructions from Sections
\ref{preliminaries} and \ref{TopologyFiberBundle}, we are ready to
move on with the definition of the desired Poincar\'e map. Our
first step will be to set up a few domains in $\Disc$ that will
play an important role in the construction of the map. From this
moment on, all interiors and closures of subsets of $\Disc$ will
be relative to the topology of the open disc $\Disc$. Lift the
domain $ \Aprime $ onto $\Disc$ to obtain $\Aprimehat=
\pi^{-1}(\Aprime).$ Take $\Cdeltaprimehat$ to be the connected
component of $\pi^{-1}(\Cdeltaprime)$ that contain the point
$z_0.$ Define the compact domain $Q^{\prime}=Q \cap
\pi^{-1}(\overline{\Cdeltaprime}).$ We can think of $Q^{\prime}$
as the ideal geodesic polygon $Q$ with its corners cut out along
horocycle arcs. Attach to $Q^{\prime}$ the neighboring congruent
pieces to form the compact domain
$$\Cprimehat= \cup \{\gamma(Q^{\prime}) \,:\,
\gamma \in \{id_{\Disc},\gamma_1,...,\gamma_{n^2},
\gamma^{-1}_1,...,\gamma^{-1}_{n^2}\}\}.$$ Similarly, let $Q_A=Q
\cap \pi^{-1}(\overline{A})$ and let
$$\CAhat = \cup \{\gamma(Q_A) \,:\,
\gamma \in \{id_{\Disc},\gamma_1,...,\gamma_{n^2},
\gamma^{-1}_1,...,\gamma^{-1}_{n^2}\}\}.$$
 If we denote by $\Chat$ the intersection
$Q \cap \pi^{-1}(\overline{\Cdelta}),$ then by construction $\Chat
\subset \CAhat \subset \Cprimehat \subset \Aprimehat.$

In the constructions that are going to follow we will need the
group $\Gamma_0 = \langle \gamma_j \,\, | \,\, j \in J(\delta_0)
\rangle$ and its lift $\hat{\Gamma}_0 = \langle
\hat{\gamma}_j=\gamma_j \times D_{\gamma_j} \,\, | \,\, j \in
J(\delta_0) \rangle$ which are subgroups of the deck groups
$\Gamma$ and $\hat{\Gamma}$ respectively. With the help of those
groups we define the closed domains
$$\Xdeltahat = \cup_{\gamma \in \Gamma_0} \gamma(\Chat), \,\,\,
\Xdeltaprimehat = \cup_{\gamma \in \Gamma_0} \gamma(\Cprimehat)
\,\,\,\text{and}\,\,\,\Ahat = \cup_{\gamma \in \Gamma_0}
\gamma(\CAhat).$$ Notice, that $\Xdeltahat$ is in fact the closure
of $\Cdeltahat$.

%%%%%%%%%%%%%%%%%%%%%%%%%%%%%%%%%%%%%%%%%%%%%%%%%%%%%%%%%%%%%%%%
% PULLBACK OF THE FOLIATION

Consider the pull-back $\hat{\Fol}=\Pi^*\Fol.$ This is a foliation
on $\Disc \times S$ invariant with respect to the action of
$\hat{\Gamma}.$ In other words, if $\hat{\gamma} \in \hat{\Gamma}$
and $\hat{\varphi}^{\e}_{(z,p)}$ is a leaf of the foliation
$\hat{\Fol}$ passing through the point $(z,p) \in \Disc \times S,$
then $\hat{\gamma} (\hat{\varphi}^{\e}_{(z,p)}) =
\hat{\varphi}^{\e}_{\hat{\gamma}(z,p)}.$ Notice that the closure
of the projection $\Pi(\Aprimehat \times
\{p_0\})=\Apzeroprime$ %and $\Pi(\Cdeltaprime \times
%\{p_0\})=\Cpzeroprime,$
is compact in $E$ and thus, the line field of the foliation $\Fol$
is transverse to $\Apzeroprime$ for all $|\e| \leq r,$ where $r>0$
is small enough.

%%%%%%%%%%%%%%%%%%%%%%%%%%%%%%%%%%%%%%%%%%%%%%%%%%%%%%%%%%%%%%%%%%%%%
% POINCAR\'E MAP CONSTRUCTION

\begin{lem} \label{lemma about lifted Pmap}
For small enough $r>0$ and for any $|\e| \leq r$ there exists a
smooth Poincar\'e map $\Pmaphat : \Cprimehat \times \{p_0\} \to
\Aprimehat \times \{p_0\}$ associated with the foliation
$\hat{\Fol}$ such that for any $\hat{\gamma} \in \hat{\Gamma}$ if
both points $(z,p_0)$ and $\hat{\gamma}(z,p_0)$ belong to
$\Cprimehat \times \{p_0\}$ then $ \hat{\gamma} \circ \Pmaphat =
\PmaphatD \circ \hat{\gamma}.$ In particular, if $\hat{\gamma} \in
\hat{\Gamma}_0$ then $ \hat{\gamma} \circ \Pmaphat = \Pmaphat
\circ \hat{\gamma}.$ Moreover, for an integer $m>0$ the radius
$r>0$ can be chosen small enough so that $\Pmaphat^k(\Chat \times
\{p_0\}) \subset \CAhat \times \{p_0\},$ for $k=1,...,m$ and for
all $\e \in D_{r}(0)$
\end{lem}

\begin{proof} As usual, let $pr_S : \Disc \times S \to S$ be the
projection $(z,p) \mapsto p.$ By continuous dependance of
$\hat{\Fol}$ on parameters and initial conditions, we can choose
the radius $r$ of the parameter space so that the construction
that follows holds for any $|\e| \leq r$. Choose an arbitrary
point $(z,p_0) \in \Cprimehat \times \{p_0\}.$ If
$\hat{\varphi}^{\e}_{(z,p_0)}$ is the leaf of the perturbed
foliation $\hat{\Fol},$ passing through $(z,p_0),$ lift the loop
$\delta_0$ to a curve $\hat{\delta}_{\e}(z,p_0)$ on
$\hat{\varphi}^{\e}_{(z,p_0)}$ so that $\hat{\delta}_{\e}(z,p_0)$
covers $\delta_0$ under the projection $pr_S.$ Since $r$ is chosen
small enough, the lift $\hat{\delta}_{\e}(z,p_0)$ is contained in
the domain $\Aprimehat \times S$ and both of its endpoints are on
$\Aprimehat \times \{p_0\}.$ The first endpoint is $(z,p_0) \in
\Cprimehat \times \{p_0\}$ and the second we denote by
$\Pmaphat(z,p_0)=(\Pmaptilde(z),p_0) \in \Aprimehat \times
\{p_0\}.$ Thus, we obtain the correspondence $\Pmaphat :
\Cprimehat \times \{p_0\} \to \Aprimehat \times \{p_0\},$ which is
a smooth map close to identity. Notice, that for some integer $m >
0$ if we decrease the radius of the parameter space enough, then
by continuous dependance on parameters and initial conditions we
can make sure that for any $\e \in D_{r}(0),$ all $m$ iterations
of $\Chat \times \{p_0\}$ under $\Pmaphat$ fall inside $\CAhat
\times \{p_0\}.$

By construction, the cross-section $\Aprimehat \times \{p_0\}$ is
$\hat{\Gamma}-$invariant. Now, assume $(z,p_0) \in \Cprimehat
\times \{p_0\}$ is such that $\hat{\gamma}(z,p_0)=(\gamma(z),p_0)
\in \Cprimehat \times \{p_0\}$ for some $\hat{\gamma} \in
\hat{\Gamma}.$ As pointed out earlier, the arc
$\hat{\delta}_{\e}(z,p_0)$ is the lift of $\delta_0$ on
$\hat{\varphi}^{\e}_{(z,p_0)}$ under the projection $pr_S.$ It
connects the two points $(z,p_0) \in \Cprimehat \times \{p_0\}$
and $\Pmaphat(z,p_0) \in \Aprimehat \times \{p_0\}.$ The image
$\hat{\gamma}(\hat{\delta}_{\e}(z,p_0))$ lies on the leaf
$\hat{\varphi}^{\e}_{\hat{\gamma}(z,p_0)}$ and its endpoints are
$\hat{\gamma}(z,p_0) \in \Cprimehat \times \{p_0\}$ and
$\hat{\gamma}(\Pmaphat(z,p_0)) \in \Aprimehat \times \{p_0\}.$ We
can see that $pr_S \circ
\hat{\gamma}(z,p)=pr_S(\gamma(z),D_{\gamma}(p))=D_{\gamma}(p)=
D_{\gamma} \circ pr_S(z,p).$ %we can conclude that $pr_S \circ
%\hat{\gamma}=D_{\gamma} \circ pr_S.$
The fact that $\hat{\delta}_{\e}(z,p_0)$ is the lift of $\delta_0$
on the leaf $\hat{\varphi}^{\e}_{(z,p_0)}$ from $\Fol$ means that
$pr_S(\hat{\delta}_{\e}(z,p_0))=\delta_0.$ Similarly, to find out
what the arc $\hat{\gamma}(\hat{\delta}_{\e}(z,p_0))$ is a lift of
we just have to project it onto $S.$ Using the property $pr_S
\circ \hat{\gamma}=D_{\gamma} \circ pr_S$ we conclude that $pr_S
\circ \hat{\gamma}(\hat{\delta}_{\e}(z,p_0))=D_{\gamma} \circ
pr_S(\hat{\delta}_{\e}(z,p_0))=D_{\gamma}(\delta_0).$ That is,
$\hat{\gamma}(\hat{\delta}_{\e}(z,p_0))$ is the lift of
$D_{\gamma}(\delta_0)$ on the leaf
$\hat{\varphi}^{\e}_{\hat{\gamma}(z,p_0)}$ under the projection
$pr_S.$ Therefore, the endpoint $\hat{\gamma}(\Pmaphat(z,p_0))$
can also be represented as $\PmaphatD(\hat{\gamma}(z,p_0)).$ Thus,
we obtain the relation $\hat{\gamma} \circ \Pmaphat = \PmaphatD
\circ \hat{\gamma}.$

The base loop $\delta_0 \subset S$ is chosen so that whenever
$\delta_0 \cdot \delta_j=0$ then $\delta_0 \cap
supp(D_{\gamma_j})=\varnothing.$ Because of this choice, if
$\gamma \in \Gamma_0$ we have the identity
$D_{\gamma}(\delta_0)=\delta_0.$ That leads to the second
equivariance relation $\hat{\gamma} \circ \Pmaphat = \Pmaphat
\circ \hat{\gamma}.$
\end{proof}

Lemma \ref{lemma about lifted Pmap} allows us to extend $\Pmaphat$
from a map on $\Cprimehat \times \{p_0\}$ to a $\hat{\Gamma}_0$ -
equivarint map on the cross-section $\Xdeltaprimehat \times
\{p_0\}.$ In particular, since $\Cdeltaprimehat \times \{p_0\}$ is
a $\hat{\Gamma}_0-$invariant open subdomain of $\Xdeltaprimehat
\times \{p_0\},$ the map $\Pmaphat$ is well defined and
$\hat{\Gamma}_0-$equivarint on it. This fact makes it possible for
the $\Pmaphat$ to descend under the covering $\Pi$ to a Poincar\'e
map defined on $\Cpzeroprime.$

\begin{cor} \label{Corollary Pmaphat extension}
The transformation $\Pmaphat$ constructed in lemma \ref{lemma
about lifted Pmap} gives rise to a map $\Pmaphat : \Xdeltaprimehat
\times\{p_0\} \to \Aprimehat \times \{p_0\}$ for the foliation
$\hat{\Fol}$ such that for any $\hat{\gamma} \in \hat{\Gamma}_0$
the equivariance relation $ \hat{\gamma} \circ \Pmaphat = \Pmaphat
\circ \hat{\gamma}$ holds. In particular, the restriction of
$\Pmaphat$ on $\Cdeltaprimehat \times \{p_0\}$ satisfies the same
equivarance relation $ \hat{\gamma} \circ \Pmaphat = \Pmaphat
\circ \hat{\gamma}$ for $\hat{\gamma} \in \hat{\Gamma}_0.$
\end{cor}

\begin{proof}
Notice that $\Gamma_0$ keeps both domains $\Xdeltaprimehat$ and
$\Cdeltaprimehat$ invariant. In other words,
$\gamma(\Xdeltaprimehat)=\Xdeltaprimehat$ and
$\gamma(\Cdeltaprimehat)=\Cdeltaprimehat$ for any $\gamma \in
\Gamma_0.$ This immediately leads to the invariance of the
cross-sections $\Xdeltaprimehat \times \{p_0\}$ and
$\Cdeltaprimehat \times \{p_0\}$ under the action of
$\hat{\Gamma}_0.$

Since $\Xdeltaprimehat = \cup_{\gamma \in \Gamma_0}
\gamma(\Cprimehat),$ we can define $\Pmaphat$ on
$\gamma(\Cprimehat) \times \{p_0\} = \hat{\gamma}(\Cprimehat
\times \{p_0\})$ as the conjugated map $\hat{\gamma} \circ
\Pmaphat \circ \hat{\gamma}^{-1} : \gamma(\Cprimehat) \times
\{p_0\} \to \Aprimehat \times \{p_0\}.$ By lemma \ref{lemma about
lifted Pmap}, for $\hat{\gamma}_1$ and $\hat{\gamma}_2 \in
\hat{\Gamma}_0,$ the two maps $\hat{\gamma}_1 \circ \Pmaphat \circ
\hat{\gamma}_1^{-1}$ and $\hat{\gamma}_2 \circ \Pmaphat \circ
\hat{\gamma}_2^{-1}$ agree on the intersection
$\hat{\gamma}_1(\Cprimehat \times \{p_0\}) \cap
\hat{\gamma}_2(\Cprimehat \times \{p_0\})$ whenever it is
nonempty. As $\Cdeltaprime$ is a $\Gamma_0-$invariant subdomain of
$\Xdeltaprimehat,$ the second statement follows immediately.
\end{proof}

\begin{cor} \label{Corollary Pmaphat descend}
The transformation $\Pmaphat : \Cdeltaprimehat \times \{p_0\} \to
\Aprimehat \times \{p_0\}$ associated with the foliation
$\hat{\Fol}$ descends to a smooth Poincar\'e map $\Pmap :
\Cpzeroprime \to \Apzeroprime$ for the foliation $\Fol$ under the
covering bundle map $\Pi : \dtimess \to E$. In other words, for
any $(z,p_0) \in \Cdeltaprimehat \times \{p_0\}$ the relation $\Pi
\circ \Pmaphat(z,p_0) = \Pmap \circ \Pi(z,p_0)$ holds.
\end{cor}

\begin{proof}
The statement follows directly from corollary \ref{Corollary
Pmaphat extension}.
\end{proof}

At this point, it is not difficult to explain the role of the
index set $J(\delta_0)$ and the choice of the cuts in the
definition of $\Bdelta$ and subsequently of $\Cpzero$ and
$\Cpzeroprime$. Whenever $j \in J(\delta_0),$ the loop $\delta_0$
does not intersect the vanishing cycle $\delta_j$ and in fact is
contained in $S - supp(D_{\gamma_j}).$ Hence, it is true that
$D_{\gamma_j}(\delta_0)=\delta_0.$ As a result of this, the
descended map $\Pmap$ is univalent around the hole in
$\Cpzeroprime$ associated to the singularity $a_j.$ On the other
hand, for $i$ not in $J(\delta_0)$ the loop $\delta_0$ intersects
$\delta_i$ and so $D_{\gamma_i}(\delta_0)$ is not even free
homotopic to $\delta_0.$ Therefore the map $\Pmap$ is going to
branch switching from $\Pmap$ to $\PmapD$ when going through a
cut.

On a side note, but still worth mentioning is a fact that follows
from the constructions in the proof of lemma \ref{lemma about
lifted Pmap}. It is not difficult to see that the Poincar\'e map
does not change when the base loop $\delta_0$ has been homotoped
appropriately. In other words, if $\delta_0$ is homotopic on $S$
to another loop $\delta_0^{'}$ passing through $p_0,$ then the two
maps $\hat{P}_{\delta_0, \e}$ and $\hat{P}_{\delta^{'}_0, \e}$
will be equal, as long as $\delta_0^{'}$ is close enough to
$\delta_0$ on $S$ or the radius $r$ is kept small enough. Thus, if
we slightly wiggle $\delta_0$ on $S$ so that the base point $p_0$
is kept fixed, the resulting Poincar\'e map will stay the same.
This provides us with the opportunity to adjust the loop
$\delta_0$ if necessary. The same is true for $\Pmap.$

\subsection{Complex Structures on the Cross-Section} \label{Section
Complex Structures}

Apart from the smooth structure of a fiber bundle, the space $E$,
being a subset of $\CC,$ has a complex structure with respect to
which the foliation $\Fol$ is holomorphic and depends analytically
on the parameter $\e$. This fact provides the foliation with very
specific properties. On the other hand, the Poincar\'e map $\Pmap
: \Cpzeroprime \to \Apzeroprime$ for the perturbed foliation
$\Fol$ captures some topological properties of the foliation.
Since some of those properties are strongly related to the
holomorphic nature of the foliation, we would like our Poincar\'e
map to reflect the complex analyticity of $\Fol.$ So far $\Pmap$
is defined as a smooth map on the smooth surface $\Cpzeroprime$
and therefore our next step is to induce a complex structure on
$\Cpzeroprime$ in which the Poincar\'e transformation is
holomorphic.

Since the closure of $\Apzeroprime$ is transverse to $\Fol$, there
is an open neighborhood $\Apzerotilde$ of $\Apzeroprime$ such that
$\Apzerotilde$ is transverse to $\Fol.$ Fix $\e \in D_r(0).$ Take
a point $q_0 \in \Apzerotilde$ and a complex cross-section
$\Lqzero$ through $q_0,$ transverse to $\Fol.$ More precisely,
$\Lqzero$ is a complex segment, that is, it lies on a complex line
through $q_{0}$ and is a real two dimensional disc.

The fact that the foliation $\Fol$ is holomorphic and
$\Apzerotilde$ is smoothly embedded surface transverse to $\Fol$
provides us with convenient flow-box charts. A chart of this kind
consists of an open neighborhood $FB(q_0) \subset E$ of $q_0$ and
a biholomorphic map $$\betaqzeroe \, : \, \Disc \times \Disc
\longrightarrow FB(q_0)$$ with the following properties:
\vspace{1.5mm}

\noindent 1. \, $\betaqzeroe(0,0)=q_0;$ \vspace{1.5mm}

\noindent 2. \, $\betaqzeroe(\{\zeta\} \times \Disc)$ is a
connected component of the intersection of $FB(q_0)$  with the
leaf $\phie_{\betaqzeroe(\zeta,0)}$ through the point
$\betaqzeroe(\zeta,0)$ for any $\zeta \in \Disc;$ \vspace{1.5mm}

\noindent 3. \, $\betaqzeroe(\Disc \times \{0\})=\Lqzero;$
\vspace{1.5mm}

\noindent 4. \, The portion of $\Apzerotilde$ passing through
$FB(q_0)$ looks like the graph of a smooth map $\alphaqzeroe :
\Disc \to \Disc$ in the chart $\Disc \times \Disc.$ In other words
\begin{align*}\betaqzeroe^{-1}(FB(q_0) \cap \Apzerotilde)=
\{(\zeta,\alphaqzeroe(\zeta)) \in \Disc \times \Disc \,\, | \,\,
\alphaqzeroe : \Disc \to \Disc \,\, \text{smooth}\}.
\end{align*}

\noindent Denote by $\Uqzero$ the open subset $FB(q_0) \cap
\Apzerotilde$ of $\Apzerotilde.$ Let $pr_j : \Disc \times \Disc
\to \Disc$ be $pr_j(\zeta_1,\zeta_2)=\zeta_j,$ where $j=1,2.$
Define the diffeomorphism
\begin{align*}
\phiqzeroe \,\, &: \,\, \Uqzero \longrightarrow \Disc \,\,\, \text{by} \\
\phiqzeroe \,\, &: \,\, q \longmapsto pr_1 \circ (\betaqzeroe^{-1})\big{|}_{\Uqzero}(q) \\
\phiqzeroe^{-1} \,\, &: \,\, \zeta \longmapsto
\betaqzeroe(\zeta,\alphaqzeroe(\zeta)).
\end{align*}
Consider the family of pairs
$\Atlas_{\e}(\Apzerotilde)=\{(\Uqzero,\phiqzeroe)\, | \, q_0 \in
\Apzerotilde\}.$

\begin{lem} \label{lemma complex atlas}
The collection of charts $\Atlas_{\e}(\Apzerotilde)$ is a
holomorphic atlas for the surface $\Apzerotilde.$
\end{lem}

\begin{proof}
Let $q_1, q_2$  be two points from the surface $\Apzerotilde$
with chart neighborhoods $\Uqone \cap \Uqtwo \neq \varnothing.$
Let $V_j = \phiqje(\Uqone \cap \Uqtwo)$ for $j=1,2.$ Consider the
diffeomorphism $\phiqtwoe \circ \phiqonee^{-1} : V_1 \to V_2.$ For
a point $\zeta \in V_1$ compute
\begin{align*}
\phiqtwoe \circ \phiqonee^{-1}(\zeta) &= pr_1 \circ
(\betaqtwoe^{-1})\big{|_{\Uqtwo}} \circ
\betaqonee(\zeta,\alphaqonee(\zeta)) \\
&= pr_1 \circ(\betaqtwoe^{-1} \circ
\betaqonee)(\zeta,\alphaqonee(\zeta)).
\end{align*}
The map $$\betaqtwoe^{-1} \circ \betaqonee \, : \,
\betaqonee^{-1}(FB(q_1) \cap FB(q_2)) \, \longrightarrow \,
\betaqtwoe^{-1}(FB(q_1) \cap FB(q_2))$$ is a holomorphic
isomorphism. Let us take a local leaf $\{\zeta\} \times \Disc$
contained in the open set $\betaqonee^{-1}(FB(q_1) \cap FB(q_2)).$
Then, $\betaqonee(\{\zeta\} \times \Disc)$ lies on the leaf
$\phie_{\betaqonee(\zeta,0)}$ from the foliation $\Fol$. Since
$\phie_{\betaqonee(\zeta,0)}$ passes through the intersection
$FB(q_1) \cap FB(q_2),$ there exists $\zeta^{\prime} \in \Disc$
such that
$\phie_{\betaqtwoe(\zeta^{\prime},0)}=\phie_{\betaqonee(\zeta,0)}.$
It follows from here that $\betaqtwoe(\{\zeta^{\prime}\} \times
\Disc)$ lies on the leaf
$\phie_{\betaqtwoe(\zeta^{\prime},0)}=\phie_{\betaqonee(\zeta,0)}.$
Hence, $$\betaqtwoe^{-1} \circ \betaqonee(\{\zeta\} \times
\Disc)=\{\zeta^{\prime}\} \times \Disc.$$ Therefore $$pr_1 \circ
\betaqtwoe^{-1} \circ \betaqonee(\zeta,\xi)=pr_1 \circ
\betaqtwoe^{-1} \circ \betaqonee(\zeta,0)$$ for all $\xi \in
\Disc$. In particular, $$\phiqtwoe \circ \phiqonee(\zeta)=pr_1
\circ \betaqtwoe^{-1} \circ
\betaqonee(\zeta,\alphaqonee(\zeta))=pr_1 \circ \betaqtwoe^{-1}
\circ \betaqonee(\zeta,0),$$ is a holomorphic transformation with
respect to $\zeta.$ Notice, that in fact the transition map
$\phiqtwoe \circ \phiqonee(\zeta)$ depends holomorphically on $\e$
as well.
\end{proof}

The choice of complex structure on the surface $\Apzerotilde$ is
justified by the next lemma. As it turns out, the map $\Pmap$ is
holomorphic in the complex structure $\Atlas_{\e}(\Apzerotilde).$

\begin{lem} \label{Lemma Complex Pmap}
The Poincar\'e map $\Pmap : \Cdeltaprime \to \Apzeroprime$
associated to the foliation $\Fol$ is holomorphic in the complex
structure defined by the atlas $\Atlas_{\e}(\Apzerotilde)$ and
depends analytically with respect to the parameter $\e$.
\end{lem}

\begin{proof}
Let $q_1 \in \Cdeltaprime$ and $q_2 \in \Apzeroprime$ be two
points such that $q_2=\Pmap(q_1).$ Find charts $\Uqone$ and
$\Uqtwo$ such that $\Pmap(\Uqone) \subset \Uqtwo$ and $\Lqone$ and
$\Lqtwo$ are the corresponding cross-sections. According to the
definition for a holomorphic transformation with respect to a
complex atlas, $\Pmap$ is considered holomorphic whenever
$$\phiqtwoe \circ \Pmap \circ \phiqonee^{-1} : \Disc
\longrightarrow \Disc$$ is holomorphic.

For an arbitrary $q_0 \in \Apzerotilde$ define the map
\begin{align*}
  \barphiqzeroe \, &: \, \Uqzero  \longrightarrow \Lqzero,\\
  \barphiqzeroe \, &: \, q \longmapsto \betaqzeroe\bigl(pr_1 \circ
  \betaqzeroe^{-1}(q),0\bigr), \\
  \barphiqzeroe^{-1} \, &: \, q^{\prime} \longmapsto \betaqzeroe\Bigl(pr_1 \circ
  \betaqzeroe^{-1}(q^{\prime}), \alphaqzeroe\bigl(pr_1 \circ
  \betaqzeroe^{-1}(q^{\prime})\bigr)\Bigr).
\end{align*}
When $\barphiqzeroe$ is pre-composed with $\betaqzeroe^{-1},$ the
following chain of equalities holds:
\begin{align*}
\betaqzeroe^{-1} \circ \barphiqzeroe(q)&= \betaqzeroe^{-1} \circ
\betaqzeroe\bigl(pr_1 \circ \betaqzeroe(q),0\bigr)\\
&= pr_1 \circ \betaqzeroe^{-1}(q)\\
&= \phiqzeroe(q).
\end{align*}

Let us look at the smooth map $$\barphiqzeroe \circ \Pmap \circ
\barphiqzeroe^{-1} : \Lqone \to \Lqtwo.$$ As noted, $\barphiqzeroe
\circ \Pmap \circ \barphiqzeroe^{-1}(\Lqone) \subset \Lqtwo.$ For
$j=1,2$ and a point $q^{\prime} \in \Lqj,$ the image
$\zeta^{\prime}=pr_1(\betaqje^{-1}(q^{\prime}))$ belongs to
$\Disc.$ The straight segment
$$\Upsilon_{q_j,\e}=[0,\alphaqje(\zeta^{\prime})]$$ on $\Disc$
connects $0$ to the point $\alphaqje(\zeta^{\prime})$ so
$\{\zeta^{\prime}\} \times \Upsilon_{q_j,\e}$ lies on the local
leaf $\{\zeta^{\prime}\} \times \Disc.$ Therefore
$$\lambda_j^{\e}(q^{\prime})=\betaqje(\{\zeta^{\prime}\} \times
\Upsilon_{q_j,\e})$$ is an arc on $\phie_{q^{\prime}} \cap
FB(q_j)$ with one endpoint $q^{\prime} \in \Lqj$ and the second
one being
$$\betaqje\Bigl(pr_1 \circ \betaqje^{-1}(q^{\prime}),\,\alphaqje\bigl(pr_1 \circ \betaqje^{-1}(q^{\prime})\bigr)\Bigr)
= \barphiqje^{-1}(q^{\prime}) \in \Uqj.$$

Remember that the lifted Poincar\'e transformation $\Pmaphat$ was
constructed in lemma \ref{lemma about lifted Pmap} as a
correspondence between the endpoints $(\tilde{z},p_0)$ and
$\Pmaphat(\tilde{z},p_0)$ of the path
$\hat{\delta}_{\e}(\tilde{z},p_0)$. This path was obtained as the
lift of $\delta_0 \subset S$ to the leaf
$\hat{\varphi}^{\e}_{(\tilde{z},p_0)}$ of the foliation
$\hat{\Fol}$ under the projection $pr_S.$ Let
$\delta_{\e}(\tilde{q})=\Pi(\hat{\delta}_{\e}(\tilde{z},p_0)),$
where $\tilde{q}=\Pi(\tilde{z},p_0) \in \Cpzeroprime.$ Consider
the path
$$\lambda^{\e}(q^{\prime})=\lambda^{\e}_1(q^{\prime}) \cdot \delta_{\e}\bigl(\barphiqonee^{-1}(\tilde{q})\bigr)
\cdot \Bigl(\lambda^{\e}_2\bigl(\barphiqtwoe \circ \Pmap \circ
\barphiqonee^{-1}(q^{\prime})\bigr)\Bigr)^{-1}.$$ The path
connects the point $q^{\prime} \in \Lqone$ to the point
$P_{q_1,q_2,\e}(q^{\prime})=\barphiqtwoe \circ \Pmap \circ
\barphiqonee^{-1}(q^{\prime}).$ By construction,
$\lambda^{\e}(q^{\prime})$ lies on the leaf $\phie_{q^{\prime}}$
and varies continuously with respect to both the endpoint
$q^{\prime} \in \Lqone$ and the parameter $\e \in D_r(0).$ The
other endpoint $P_{q_1,q_2,\e}(q^{\prime})$ belongs to the
intersection $\varphi^{\e}_{q^{\prime}} \cap \Lqtwo.$ As we
already know, $\Lqone$ and $\Lqtwo$ are holomorphic cross-sections
and $\varphi^{\e}_{q^{\prime}}$ is a leaf of the holomorphic
foliation $\Fol$ depending analytically on $\e.$ Then, by analytic
dependence of the foliation on parameters and initial conditions
\cite{IY}, it follows that $P_{q_1,q_2,\e}(q^{\prime})$ depends
analytically on $(q^{\prime},\e).$ In other words, the map
$$P_{q_1,q_2,\e}(q^{\prime})=\barphiqtwoe \circ \Pmap \circ \barphiqonee^{-1} \, : \,
\Lqone \longrightarrow \Lqtwo$$ is a holomorphic map depending
holomorphically on $\e.$ Conjugating with the holomorphic maps
$\betaqonee$ and $\betaqtwoe$ we conclude that
$$(\betaqtwoe^{-1})|_{\Lqtwo} \circ \barphiqtwoe \circ \Pmap \circ \barphiqonee^{-1}
\circ (\betaqonee)|_{(\{0\}\times \Disc)}= \phiqtwoe \circ \Pmap
\circ \phiqonee^{-1} : \Disc \rightarrow \Disc$$ is also
holomorphic and depends analytically on $\e.$
\end{proof}

\subsection{Periodic Orbits and Complex Cycles}

We proceed with the study of the Poincar\'e maps $\Pmap$ and
$\Pmaphat.$ More precisely, we are interested in the relationship
between their periodic orbits and the complex cycles of the
perturbed foliation $\Fol.$

First, we start with a more general result.

\begin{lem} \label{Lemma_periodic_orbits_cycles_lifted}
Let $r>0$ be the radius obtained in lemma \ref{lemma about lifted
Pmap}. Let $$\Pmaphat : \Xdeltaprimehat \times \{p_0\} \to
\Aprimehat \times \{p_0\}$$ be the map defined in corollary
\ref{Corollary Pmaphat extension}, where $\e \in D_r(0)$. Then,
the following statements are true: \vspace{1.5mm}

\noindent {\em1.} Assume $\Pmaphat$ has a periodic orbit
$((z_1,p_0),...,(z_m,p_0))$ in $\Xdeltaprimehat \times \{p_0\}.$
Then the foliation $\Fol$ has a marked complex cycle
$(\Delta_{\e}, q_{\e})$ with a base point $q_{\e}=\Pi(z_1,p_0)$
and a representative $\delta_{\e}$ contained in $\EAprime.$
\vspace{1.5mm}

\noindent {\em 2.} For an arbitrary representative
$\delta^{\prime}_{\e}$ of the marked complex cycle $(\Delta_{\e},
q_{\e})$, if $\delta^{\prime}_{\e}$ is contained in $\Edelta$ then
it is $\Dgamma(\delta_0),m-$fold vertical for some $\gamma \in
\Gamma$. Moreover, if $z_1$ belongs to $\Cdeltaprimehat \subset
\Xdeltaprimehat,$ then $\gamma \in \Gamma_0$ and thus,
$\delta^{\prime}_{\e}$ is $\delta_0,m-$fold vertical. Otherwise,
if $z_1$ is in $\Xdeltaprimehat - \Cdeltaprimehat,$ then $\gamma
\in \Gamma-\Gamma_0$ and therefore $\delta^{\prime}_{\e}$ is not
$\delta_0,m-$fold vertical.
\end{lem}

\begin{proof}
 Consider the map $\Pmaphat : \Xdeltaprime \times
\{p_0\} \to \Aprimehat \times \{p_0\}$ and let its orbit
$(z_1,p_0)$, ...,$(z_m,p_0)$ be periodic on $\Xdeltaprimehat
\times \{p_0\}.$ For convenience, let $(z_{m+1},p_0)=(z_1,p_0).$
Notice that since all $m$ points belong to the same orbit, they
lie on the same leaf $\hat{\varphi}^{\e}_{(z_1,p_0)}$ from the
foliation $\hat{\Fol}.$ Let $\delta(z_i,z_{i+1}),$ for
$i=1,...,m,$ be the lift of $\delta_0$ on the leaf
$\hat{\varphi}^{\e}_{(z_1,p_0)}$ so that $\delta(z_i,z_{i+1})$
covers $\delta_0$ under the projection $pr_S$ and connects the
points $(z_i,p_0)$ and $(z_{i+1},p_0).$ By the construction of the
map $\Pmaphat$ in the proof of lemma \ref{lemma about lifted
Pmap}, all arcs $\delta(z_i,z_{i+1})$ are contained in $\Aprimehat
\times S.$ Therefore, the path $\hat{\delta}_{\e} =
\cup_{i=1}^{m-1} \delta(\hat{q}_i,\hat{q}_{i+1})$ is contained in
$\Aprimehat \times S$ and goes through all the points
$(z_1,p_0),...,(z_m,p_0).$ Also, its two endpoints are $(z_1,p_0)$
and $(z_{m+1},p_0)=(z_1,p_0)$ so in fact $\hat{\delta}_{\e}$ is a
loop.

When mapping $\hat{\delta}_{\e}$ with $\Pi$ back onto $E$ we
obtain a loop $\delta_{\e}=\Pi(\hat{\delta}^{\e})$ lying on the
leaf $\varphi^{\e}_{q_{\e}}=\Pi(\hat{\varphi}^{\e}_{(z_1,p_0)})$
from the perturbed foliation $\Fol.$ Moreover, $\delta_{\e}$ is
contained in $\EAprime=\Pi(\Aprimehat \times S).$ As discussed in
\cite{PL55} and \cite{PL57}, the loop $\delta_{\e}$ is non trivial
on $\varphi^{\e}_{q_{\e}}$ and defines a marked complex cycle
$(\Delta_{\e},q_{\e}).$

Let us now look at an arbitrary representative
$\delta^{\prime}_{\e}$ of the marked complex cycle $(\Delta_{\e},
q_{\e})$ and let us assume $\delta^{\prime}_{\e}$ is contained in
$\Edelta.$ By assumption, $\delta_{\e}^{\prime}$ and $\delta_{\e}$
are representatives of the same marked cycle
$(\Delta_{\e},q_{\e})$ for the foliation $\Fol.$ This implies that
there exists a homotopy $\delta(t)$ on the leaf
$\varphi^{\e}_{q_{\e}}$ between the two loops, keeping the base
point $q_{\e}$ fixed. Since the leaf is contained in $E,$ the
homotopy $\delta(t)$ takes place inside $E.$ As pointed out
earlier, $\delta_{\e}$ lifts to a loop $\hat{\delta}_{\e}$
contained in $\Aprimehat \times S$ and passing through
$(z_1,p_0)$. By the homotopy lifting property for covering spaces
\cite{H}, $\delta(t)$ lifts to a homotopy $\hat{\delta}(t)$ inside
$\Disc \times S$ so that $\Pi(\hat{\delta}(t))=\delta(t).$ Since
$\hat{\delta}(0)= \hat{\delta}_{\e}$ is a loop, then
$\hat{\delta}(1)$ is also a loop that passes through $(z_1,p_0)$
and $\Pi(\hat{\delta}(1))=\delta(1)=\delta_{\e}^{\prime}.$ Let
$\hat{\delta}^{\prime}_{\e}=\hat{\delta}(1).$ Thus,
$\hat{\delta}^{\prime}_{\e}$ is homotopic inside $\Disc \times S$
to $\hat{\delta}_{\e}$ via $\hat{\delta}(t)$ relative to the base
point $(z_1,p_0).$

It follows from the notations in Section
\ref{Section_Vertical_Cycles} that $\Pi(\gamma(\Bdelta) \times
S)=\Edelta$ for any $\gamma \in \Gamma.$ Since
$\delta_{\e}^{\prime}$ is contained in $\Edelta,$ the loop
$\hat{\delta}^{\prime}_{\e}$ is contained in $\gamma(\Bdeltahat)
\times S,$ where $\gamma$ is chosen so that $z_1 \in
\gamma(\Bdeltahat).$ Notice that $\gamma(\Bdeltahat)=\Bdeltahat$
if and only if $\gamma \in \Gamma_0.$ Consider the following
deformation retractions
\begin{eqnarray*}
R^{\prime}_{\gamma}&=&\hat{\gamma} \circ R^{\prime} \circ
\hat{\gamma}^{-1} \, : \, \Disc \times S \, \longrightarrow \,
\{\gamma(z_0)\} \times S \,\,\,\,\, \text{and} \\
R_{\gamma} &=& \hat{\gamma} \circ R \circ \hat{\gamma}^{-1} \, :
\, \gamma(\Bdeltahat) \times S \, \longrightarrow \,
\{\gamma(z_0)\} \times S,
\end{eqnarray*}
where $R^{\prime}$ and $R$ are defined in Section
\ref{Section_Vertical_Cycles}. Then,
$R^{\prime}_{\gamma}(\hat{\delta}(t)) = \{\gamma(z_0)\} \times
pr_S(\hat{\delta}(t))$ is a homotopy on $\{\gamma(z_0)\} \times S$
between the loops $\{\gamma(z_0)\} \times pr_S(\hat{\delta}_{\e})$
and $\{\gamma(z_0)\} \times pr_S(\hat{\delta}^{\prime}_{\e}).$ By
construction, $pr_S(\hat{\delta}_{\e})=\delta_0^m.$ Therefore,
\begin{eqnarray*}
   \Pi(\{\gamma(z_0)\} \times pr_S(\hat{\delta}_{\e})) &=& \Pi \circ
        \hat{\gamma}(\{z_0\} \times \Dgamma^{-1}(\delta_0^m)) \\
        &=& \Pi(\{z_0\} \times \Dgamma^{-1}(\delta_0^m)) \\
        &=& \Dgamma^{-1}(\delta_0^m) \,\,\,\, \text{and} \\
   \Pi(\{\gamma(z_0)\} \times pr_S(\hat{\delta}_{\e}^{\prime})) &=&
        \Pi \circ \hat{\gamma}(\{z_0\} \times \Dgamma^{-1}\circ
        pr_S(\hat{\delta}_{\e}^{\prime}))\\
        &=& \Dgamma^{-1}(pr_S(\hat{\delta}_{\e}^{\prime})).
\end{eqnarray*}
are homotopic on the fiber $S.$ With the help of the fact that the
loop $\hat{\delta}^{\prime}_{\e}$ is contained in
$\gamma(\Bdeltahat) \times S,$ we deduce that $\{\gamma(z_0)\}
\times pr_S(\hat{\delta}^{\prime}_{\e}) =
R^{\prime}_{\gamma}(\hat{\delta}^{\prime}_{\e}) =
R_{\gamma}(\hat{\delta}^{\prime}_{\e}).$ But $R_{\gamma}$ is a
deformation retraction of $\gamma(\Bdeltahat) \times S$ onto
$\{\gamma(z_0)\} \times S,$ so $\hat{\delta}^{\prime}_{\e}$ is
free homotopic to $\{\gamma(z_0)\} \times
pr_S(\hat{\delta}^{\prime}_{\e})$ inside $\gamma(\Bdeltahat)
\times S.$ This fact immediately implies that
$\delta^{\prime}_{\e}=\Pi(\hat{\delta}^{\prime}_{\e})$ is free
homotopic to
$\Dgamma^{-1}(pr_S(\hat{\delta}^{\prime}_{\e}))=\Pi(\{\gamma(z_0)\}
\times pr_S(\hat{\delta}^{\prime}_{\e}))$ inside $\Edelta =
\Pi(\gamma(\Bdeltahat) \times S).$ Therefore,
$\delta^{\prime}_{\e}$ is free homotopic to
$\Dgamma^{-1}(\delta_0^m)$ inside $\Edelta.$ Since $z_1$ is from
$\Xdeltaprimehat,$ there are two options. Either $z_1 \in
\Cdeltaprimehat \cap \Xdeltaprimehat$ or $z_1 \in \Xdeltaprimehat
- \Cdeltaprimehat.$ In the first case, $\Cdeltaprimehat \subset
\Bdeltahat$ so $\gamma \in \Gamma_0$ and therefore
$\Dgamma^{-1}(\delta_0^m)=\delta_0^m$ which means that
$\delta^{\prime}_{\e}$ is $\delta_0,m-$fold vertical. In the
second case, due to the identity $\Cdeltaprimehat=\Bdeltahat \cap
\Xdeltaprimehat,$ the point $z_1$ does not belong to the domain
$\Bdeltahat,$ so $\gamma \in \Gamma - \Gamma_0$ and therefore
$\Dgamma^{-1}(\delta_0^m)$ is not even free homotopic to
$\delta_0^m$ on the fiber $S$ which implies that
$\delta^{\prime}_{\e}$ is not $\delta_0,m-$fold vertical.
\end{proof}

The lemma above leads to a corollary that settles part of Theorem
\ref{thm-pmap}.

\begin{cor}\label{corollary about peridoic orbits and cycles}
Let $r>0$ be the radius obtained in lemma \ref{lemma about lifted
Pmap}. Let $\Pmap : \Cpzeroprime \to \Apzeroprime$ be the
Poincar\'e map for $\Fol$ as described in corollary \ref{Corollary
Pmaphat descend}, where $\e \in D_r(0)$. Then, the following
statements are true:

\noindent {\em1.} If $\Pmap$ has a periodic orbit of period $m$ in
$\Cpzeroprime$ then the foliation $\Fol$ has a marked complex
cycle $(\Delta_{\e}, q_{\e})$ with a base point $q_{\e}$ belonging
to $\Cpzeroprime.$\vspace{2mm}

\noindent {\em 2.} The marked complex cycle $(\Delta_{\e},
q_{\e})$ has a representative $\delta_{\e}$ contained in
$\EAprime$ and passing through the points of the $m-$periodic
orbit.\vspace{2mm}

\noindent {\em 3.} If $\delta^{\prime}_{\e}$ is an arbitrary
representative of the marked complex cycle $(\Delta_{\e},
q_{\e})$, then $\delta^{\prime}_{\e}$ is contained in $\Edelta$
and is $\delta_0,m-$fold vertical if and only if its image
$H(\delta_{\e})$ is contained in $\Bdelta$ and is free homotopic
to a point inside $\Bdelta.$
\end{cor}

\begin{proof}
 Let us assume that the map $\Pmap : \Cpzeroprime
\to \Apzeroprime$ has a periodic orbit of period $m>0$ on
$\Cpzeroprime.$ Denote this orbit by $q_1,...,q_m.$ Consider its
lift $\hat{q}_1,...,\hat{q}_{m+1}$ on $\Cdeltaprimehat \times
\{p_0\}$ so that $\Pmaphat(\hat{q}_i)=\hat{q}_{i+1}$ for
$i=1,...,m$. Then, there exists $\hat{\gamma} \in \hat{\Gamma}_0$
such that
$\Pmaphat(\hat{q}_m)=\hat{q}_{m+1}=\hat{\gamma}(\hat{q}_{1}).$ The
fact that all $m+1$ points belong to the same orbit implies that
they lie on the same leaf $\hat{\varphi}^{\e}_{\hat{q}_1}$ from
the foliation $\hat{\Fol}.$ Analogously to the proof of lemma
\ref{Lemma_periodic_orbits_cycles_lifted}, let
$\delta(\hat{q}_i,\hat{q}_{i+1})$ be the lift of $\delta_0$ on the
leaf $\hat{\varphi}^{\e}_{\hat{q}_1}$ so that
$\delta(\hat{q}_i,\hat{q}_{i+1})$ covers $\delta_0$ under the
projection $pr_S$ and connects the points $\hat{q}_i$ and
$\hat{q}_{i+1}$ for $i=1,...,m.$ Because of the way the map
$\Pmaphat$ is defined, all arcs $\delta(\hat{q}_i,\hat{q}_{i+1})$
are contained in $\Aprimehat \times S.$ Therefore, the curve
$\hat{\delta}_{\e} = \cup_{i=1}^{m-1}
\delta(\hat{q}_i,\hat{q}_{i+1})$ is contained in $\Aprimehat
\times S$ and goes through all the points
$\hat{q}_1,...,\hat{q}_m.$

The image $\delta_{\e}=\Pi(\hat{\delta}^{\e})$ inside $E$ is a
loop lying on the leaf
$\varphi^{\e}_{q_1}=\Pi(\hat{\varphi}^{\e}_{\hat{q}_1})$ from the
perturbed foliation $\Fol.$ Moreover, $\delta_{\e}$ is contained
in $\EAprime=\Pi(\Aprimehat \times S)$ and passes through the
points of the periodic orbit $q_1,...,q_m.$ As pointed out in the
proof of the previous lemma, the loop $\delta_{\e}$ is non trivial
on $\varphi^{\e}_{q_1}$ and defines a marked complex cycle
$(\Delta_{\e},q_{\e})$, where $q_{\e}$ can be chosen to be any
point from the $m-$periodic orbit of $\Pmap.$ Without loss of
generality, we can think that $q_{\e}=q_1.$ Thus, we have proved
points 1 and 2 from the current statement.

Let us now look at an arbitrary representative
$\delta^{\prime}_{\e}$ of the marked complex cycle $(\Delta_{\e},
q_{\e})$ and its projection $H(\delta^{\prime}_{\e})$ on $B.$
Clearly, $\delta_{\e}^{\prime}$ is contained in $\Edelta$ exactly
when its image $H(\delta_{\e}^{\prime})$ is contained in
$\Bdelta.$ As we know $\Pi(\hat{q}_1)=q_1=q_{\e},$ so the loop
$\delta_{\e}^{\prime} \ni q_1$ lifts as a path
$\hat{\delta}_{\e}^{\prime}$ starting from $\hat{q}_1$ on $\Disc
\times S$ under the covering map $\Pi.$ The projection
$\tilde{\delta}_{\e}^{\prime} =
pr_{\Disc}(\hat{\delta}_{\e}^{\prime})$ on the disc $\Disc$ is the
lift of $H(\delta_{\e}^{\prime})$ under the universal covering map
$\pi.$ This is true because of the identity $H \circ \Pi = \pi
\circ pr_{\Disc}.$

Assume first that the loop $H(\delta_{\e}^{\prime})$ is contained
in $\Bdelta$ and is homotopic to a point inside $\Bdelta.$ %This
%homotopy lifts to a homotopy between a point and
%$\tilde{\delta}_{\e}^{\prime}$ inside $\Bdeltahat$ \cite{H}.
For that reason, the lift $\tilde{\delta}_{\e}^{\prime}$ is a loop
in $\Bdeltahat$ and therefore $\hat{\delta}_{\e}^{\prime}$ is also
a loop contained in $\Bdeltahat \times S$.

By assumption, $\delta_{\e}^{\prime}$ and $\delta_{\e}$ are
representatives of the same marked cycle $(\Delta_{\e},q_{\e}).$
This implies that there exists a homotopy $\delta(t)$ on the leaf
$\varphi^{\e}_{q_{\e}}$ between the two loops, keeping the base
point $q_{\e}$ fixed. Since the leaf is contained in $E,$ the
homotopy $\delta(t)$ takes place inside $E.$ As pointed out
earlier, $\delta_{\e}^{\prime}$ lifts to a loop
$\hat{\delta}_{\e}^{\prime}$ contained in $\Bdeltahat \times S$
and passing through $\hat{q}_1$. The homotopy lifting property for
covering spaces applies again \cite{H}, leading to a lifted
homotopy $\hat{\delta}(t)$ inside $\Disc \times S$ such that
$\Pi(\hat{\delta}(t))=\delta(t).$ Since $\hat{\delta}(0)=
\hat{\delta}^{\prime}_{\e}$ is a loop, then $\hat{\delta}(1)$ is
also a loop that passes through $\hat{q}_1$ and such that
$\Pi(\hat{\delta}(1))=\delta_{\e}.$ Therefore,
$\hat{\delta}(1)=\hat{\delta}_{\e}.$ It follows from here that
$\hat{q}_1 = \hat{q}_{m+1}= \hat{\gamma}(\hat{q}_1).$ But
$\hat{\gamma}$ can have a fixed point inside $\Disc \times S$ only
if $\hat{\gamma}=id_{(\Disc \times S)}.$ Therefore, the lifted map
$\Pmaphat$ has a periodic orbit of period $m$ and
$pr_{\Disc}(\hat{q}_1) \in \Cdeltaprimehat.$ By point 2 from lemma
\ref{Lemma_periodic_orbits_cycles_lifted}, it follows that the
representative $\delta^{\prime}_{\e}$ is $\delta_0,m-$fold
vertical.

It is easier to see that the converse is also true. If
$\delta_{\e}^{\prime}$ is free homotopic to $\delta_0^m$ inside
$\Edelta$ then its projection $H(\delta_{\e})$ is necessarily free
homotopic to a point inside $\Bdelta.$ If the homotopy between
$\delta_{\e}^{\prime}$ and $\delta_0^m$ is denoted by
$\delta_{\e}(t),$ then it is enough to project with $H$ and obtain
the homotopy $H(\delta_{\e}(t))$ connecting the loop
$H(\delta_{\e})$ to the point $H(\delta_0)=u_0.$
\end{proof}

\vspace{2mm}
\paragraph{\em Proof of Theorem \ref{thm-pmap}.} All pieces of the theorem
are already proved. We only need to put them together. The
existence of a global cross-section $\Bpzero$ transverse to the
unperturbed foliation $\Foliation^0$ follows from Corollary
\ref{transverse surface}. Then we can see in the beginning of
Section 4.1 that $\Apzeroprime$ is transverse to the perturbed
foliation $\Fol$. By Corollary \ref{Corollary Pmaphat descend}, we
are able to construct the desired Poincar\'e map. Lemma \ref{lemma
complex atlas} and Lemma \ref{Lemma Complex Pmap} provide us with
a complex structure on the cross-section with respect to which the
map is holomorphic. Corollary \ref{corollary about peridoic orbits
and cycles} establishes the correspondence between periodic orbits
and multi-fold cycles and explains the link between the dynamical
features of the Poincar\'e transformation and the topological
properties of the multi-fold cycles with respect to the fibred
domain $\Edelta.$  $\square$ \vspace{2mm}

\section{Rapid Evolution of Marked Complex Cycles} \label{Section
on rapid evolution }

Our next goal is to explore the behavior of multi-fold limit
cycles of $\Fol$ as the parameter $\e$ approaches zero. We would
like to show their escape from large sub-domains of the complex
plane $\CC$ as explained in Theorem \ref{rapid-evolution}. This
phenomenon is what we call a rapid evolution of marked limit
cycles and this will be the topic of the current discussion.
Before we can give a proof of Theorem \ref{rapid-evolution} we
will need some auxiliary statements.

\subsection{Continuous Families of Orbits and Cycles}

We begin with some useful constructions. Fix a positive integer
$m>0$ and for convenience, consider an embedded arc $\eta$ in the
parameter disc $D_{r}(0),$ where $r>0$ is the radius chosen in
Lemma \ref{lemma about lifted Pmap}. Define the surface
$$\Yprime=(\Aprimehat \times \{p_0\})/\hat{\Gamma}_0.$$  By
construction, $\Xdeltaprimehat \times \{p_0\},$ $\Ahat \times
\{p_0\}$ and $\Xdeltahat \times \{p_0\}$ are
$\hat{\Gamma}_0-$invariant sub-surfaces of $\Aprimehat \times
\{p_0\},$ so the quotients $$\Xdeltaprime=(\Xdeltaprimehat \times
\{p_0\})/\hat{\Gamma}_0 \, , \,\,\, Y=(\Ahat \times
\{p_0\})/\hat{\Gamma}_0 \,\,\, \text{and} \,\,\,
\Xdelta=(\Xdeltahat \times \{p_0\})/\hat{\Gamma}_0$$ are
sub-surfaces of $\Yprime$ such that $\Xdelta \subset Y \subset
\Xdeltaprimehat.$ Denote by
$$\pizero \, : \, \Aprimehat \times \{p_0\} \longrightarrow
\Yprime$$ the corresponding quotient map. Since $\Pmaphat :
\Xdeltaprimehat \times \{p_0\} \to \Aprimehat \times \{p_0\}$ is
$\hat{\Gamma}_0$ - equivariant, that is $\hat{\gamma}\circ
\Pmaphat = \Pmaphat \circ \hat{\gamma}$ for any $\hat{\gamma} \in
\hat{\Gamma}_0,$ it descends to a diffeomorphism $$\Pmaptilde \, :
\, \Xdeltaprime \, \longrightarrow \, \Yprime $$ so that $\pizero
\circ \Pmaphat = \Pmaptilde \circ \pizero.$ Because by
construction $$\Pmaphat^k(\Xdeltahat \times \{p_0\}) \subset \Ahat
\times \{p_0\}\, , \,\, \text{for $\e \in D_r(0)$ and
$k=1,...,m,$}$$ the descended map has the corresponding property
$$\Pmaptilde^k(\Xdelta) \subset Y \, , \,\, \text{for $\e \in D_r(0)$ and
$k=1,...,m$}.$$

Denote the restriction of $\Pi$ on the surface $\Aprimehat \times
\{p_0\}$ by
$$\Pipzero = \Pi|_{(\Aprimehat \times \{p_0\})} \, : \, \Aprimehat \times
\{p_0\} \longrightarrow \Apzeroprime.$$ Then, the map $\Pipzero$
is a covering map.

\begin{lem}\label{Lemma Pmaphat and Pmaptilde complex structure}
Let $\Atlas_{\e}(\Apzeroprime) = \{ (\Uqzero,\phiqzeroe) \,\, :
\,\, q_0 \in \Apzeroprime\}$ be the complex atlas for
$\Apzeroprime$ as defined in Lemma \ref{lemma complex atlas}. Then
$\Aprimehat \times \{p_0\}$ has a complex atlas
$$\Atlas_{\e}(\Aprimehat \times \{p_0\}) = \{
(\Uqzerohat,\phiqzeroehat) \, \, : \, \, \hat{q}_0 \in \Aprimehat
\times \{p_0\}\},$$ such that the covering map $\Pipzero$ is
holomorphic. The new atlas makes the lifted Poincar\'e map
$\Pmaphat$ holomorphic, depending analytically on $\e.$
Analogously, the surface $\Yprime$ has a complex structure given
by the atlas
$$\Atlas_{\e}(\Yprime) = \{(\Uxzerotilde,\phixzeroetilde) \, \, : \, \, x_0 \in \Yprime\},$$
such that the quotient map $\pizero$ is holomorphic. This new
atlas makes the map $\Pmaptilde$ holomorphic, depending
analytically on $\e.$
\end{lem}

\begin{proof}
The proof of this fact is straightforward. All we have to do is to
pull back the complex structure given by
$\Atlas_{\e}(\Apzeroprime)$ to the surface $\Aprimehat \times
\{p_0\}$ in the first case, and to push forward the same structure
on the surface $\Yprime$ in the second case.
\end{proof}

The holomorphic nature of the Poincar\'e map guarantees that every
time the map has an isolated periodic orbit for some particular
value of $\e,$ there will be a continuous family of periodic
orbits defined near that particular value of $\e.$ In other words,
an isolated periodic orbit gives rise to a local continuous family
of periodic orbits due to the complex analytic properties of the
Poincar\'e map. In addition, there will be a continuous family of
marked complex cycles as well.

\begin{lem}\label{Lemma peridoic orbits and families and connection with cycles}
Let $\eprime$ belong to the parameter disc $D_r(0),$ where the
radius $r>0$ is chosen as in Lemma \ref{lemma about lifted Pmap}.
\vspace{1.5mm}

\noindent{1.} Assume $\Pmaphateprime$ has an isolated $m-$periodic
orbit $(z_1,p_0),...,(z_m,p_0)$ on the cross-section
$\Xdeltaprimehat \times \{p_0\}.$ Then $\pizero$ maps that orbit
to an isolated $m-$periodic orbit $x_1,...,x_m$ for the map
$\Pmaptildeeprime$ on the surface $\Xdeltaprime.$ \vspace{2mm}

\noindent{2.} There exists $r^{\prime}>0$ with
$D_{r^{\prime}}(\eprime) \subset D_r(0),$ such that for any
embedded in $D_{r^{\prime}}(\eprime)$ curve $\etaprime,$ passing
through $\eprime,$ there exists a continuous family
$\bigl((z_1(\e),p_0),...,(z_m(\e),p_0)\bigr)_{\e \in \etaprime}$
of periodic orbits for the map $\Pmaphat$ on $\Xdeltaprimehat
\times \{p_0\},$ which for $\e=\eprime$ becomes
$(z_1,p_0),...,(z_m,p_0).$ Moreover, the continuous family of
$\Pmaphat$ is mapped by $\pizero$ to a continuous family of
periodic orbits $(x_1(\e),...,x_m(\e))_{\e \in \etaprime}$ for the
transformation $\Pmaptilde$ on the surface $\Xdeltaprime,$ which
for $\e=\eprime$ becomes the orbit $x_1,...,x_m.$ \vspace{2mm}

\noindent{3.} If $\Pmaphat$ has a continuous family of periodic
orbits on $\Xdeltaprimehat \times \{p_0\}$ for $\e$ varying on
some curve $\etatilde$ embedded in $D_r(0),$ then the perturbed
foliation $\Fol$ has a continuous family of marked cycles
$\{(\Delta_{\e},q_{\e})\}_{\e \in \etatilde}.$ \vspace{1mm}
\end{lem}

\begin{proof}
 By assumption, $(z_1,p_0),...,(z_m,p_0)$ is an
isolated $m-$periodic orbit of $\Pmaphateprime$ on
$\Xdeltaprimehat \times \{p_0\}.$ The image of this orbit under
the covering map $\pizero$ is denoted by $x_1,...,x_m.$ Because of
the property $\pizero \circ \Pmaphateprime = \Pmaptildeeprime
\circ \pizero,$ the orbit $x_1,...,x_m$ is also isolated and
periodic with possibly a smaller or equal period. Clearly,
$\Pmaptildeeprime^m(x_1)=\Pmaptildeeprime^m(\pizero(z_1,p_0))=\pizero
\circ \Pmaphateprime^m(z_1,p_0)=\pizero(z_1,p_0)=x_1.$

Assume there exists a smaller $k=1,...,m-1$ such that
$x_1=x_{k+1}$. Then, there exists $\hat{\gamma} \in
\hat{\Gamma}_0$ such that
$(z_{k+1},p_0)=\hat{\gamma}(z_1,p_0)=(\gamma(z_1),p_0)$ for the
corresponding $\gamma \in \Gamma_0.$ On the other hand,
$(z_{k+1},p_0)=\Pmaphateprime^k(z_1,p_0).$ Thus,
$\Pmaphateprime^k(z_1,p_0)=\hat{\gamma}(z_1,p_0).$ Applying
$\Pmaphateprime^k$ to the last equality we obtain
\begin{align*}
 \Pmaphateprime^{2k}(z_1,p_0)&=\Pmaphateprime^k \circ
 \hat{\gamma}(z_1,p_0)\\
 &=\hat{\gamma} \circ \Pmaphateprime^k(z_1,p_0)\\
 &=\hat{\gamma}^2(z_1,p_0).
\end{align*}
In general, $\Pmaphateprime^{jk}(z_1,p_0)=\hat{\gamma}^j(z_1,p_0)$
for any $j \in \Naturals.$ In particular, when $j=m$ we have
$(z_1,p_0)=\Pmaphateprime^{mk}(z_1,p_0)=\hat{\gamma}^m(z_1,p_0)=(\gamma^m(z_1),p_0).$
As it turns out, $z_1=\gamma^m(z_1)$ which means that $\gamma^m$
has a fixed point in the interior of the hyperbolic disc $\Disc.$
As a subgroup of a Fuchsian group associated to a Riemann surface,
$\Gamma_0$ can have no elliptic elements but only parabolic and
hyperbolic \cite{Hu},\cite{K}. Therefore, $\gamma^m=id_{\Disc}$
and more precisely, $\gamma=id_{\Disc}.$ Thus, as it turns out,
$(z_{k+1},p_0)=(z_1,p_0)$ which is not the case.

As $\Pmaphateprime^m(z_1,p_0)=(z_1,p_0),$ we choose a chart
$(\Uzonehat,\phizoneehat)$ form the atlas $\Atlas_{\e}(\Aprimehat
\times \{p_0\})$ around the point $(z_1,p_0)$ and a smaller
neighborhood $\Uzonehat^{\prime}$ of the same point such that
$\Uzonehat^{\prime} \subset \Uzonehat$ and
$\Pmaphateprime^m(\Uzonehat^{\prime}) \subset \Uzonehat.$ Let
$D^{\prime}=\phizoneeprimehat(\Uzonehat^{\prime}) \subset \Disc$
where $\phizoneeprimehat(z_1,p_0)=0 \in D^{\prime}.$ If
$r^{\prime}>0$ is chosen small enough, then
$$\Pem = \phizoneehat \circ \Pmaphat^m \circ \phizoneehat^{-1} \, : \, D^{\prime} \, \longrightarrow
\Disc$$ for $\e \in D_{r^{\prime}}(\eprime) \subset D_r(0).$
Notice that $\Pemprime(0)=0.$ The complex valued function
$$\Ftilde : D^{\prime} \to \cc \,\,\, \text{defined as}\,\,\,
\Ftilde(\zeta,\e)=\Pem(\zeta)-\zeta$$ is holomorphic with respect
to $\zeta \in D^{\prime}$ and with respect to $\e \in
D_{r^{\prime}}(\eprime).$ By Hartogs' Theorem \cite{GR}, it is
holomorphic with respect to $(\zeta,\e) \in D^{\prime} \times
D_{r^{\prime}}(\eprime).$ Since $\Pemprime(0)=0,$ the point
$(0,\eprime)$ is a zero of $\Ftilde,$ that is
$\Ftilde(0,\eprime)=0.$

Let us look at the zero locus of $\Ftilde$ in $D^{\prime} \times
D_{r^{\prime}}(\eprime).$ The fact that the periodic orbit is
isolated means that $(z_1,p_0)$ is an isolated fixed point for the
map $\Pmaphateprime^m$. Therefore $0$ is an isolated fixed point
for $\Pemprime$ and thus, it is an isolated zero for the
holomorphic function $\Ftilde(\zeta,\eprime).$ By Weierstrass
Preparation Theorem \cite{GR},\cite{Ch}, we can write
$$\Ftilde(\zeta,\e)=\prod_{j=1}^{s}(\zeta - \alpha_j(\e))\theta(\zeta,\e),$$
where $\theta(0,\eprime) \neq 0$ and $\{\alpha_j(\e) \, : \,
j=1,...,s\}$ depend analytically on $\e,$ satisfying the
equalities $\alpha_1(\eprime)=...=\alpha_s(\eprime)=0$ and
possibly branching into each other. %with $\alpha_1(\eprime)=...=\alpha_s(\eprime)=0.$

Now, let $\etaprime$ be some curve embedded in the disc
$D_{r^{\prime}}(\eprime)$ and passing through $\eprime.$ For $\e$
 varying on $\etaprime,$ we can choose a branch, denoted for simplicity by
$\alpha_1(\e).$ Then the desired continuous family for $\Pmaphat$
can be constructed by setting
$(z_1(\e),p_0)=\phizoneehat^{-1}(\alpha_1(\e))$ and
$(z_{j+1}(\e),p_0)=\Pmaphat^{j}(z_1(\e),p_0)$ for $j=1,...,m-1.$
Its image under the covering $\pizero$ will provide the continuous
family of periodic orbits for $\Pmaptilde.$

The third point of the statement follows directly form Lemma
\ref{Lemma_periodic_orbits_cycles_lifted} with the remark that the
representative $\delta_{\e}$ is constructed to depend continuously
on the parameter $\e.$
\end{proof}

\subsection{Proof of Theorem \ref{rapid-evolution}.}

By assumption, the Poincar\'e map $\Pmapezero$ has an isolated
periodic orbit $(q_1,...,q_m)$ on the cross-section $\Cpzeroprime$
and the perturbed foliation $\Foliation^{\e_0}$ has a marked limit
cycle $(\Delta,q_1)$ with a $\delta_0,m-$fold vertical
representative $\delta^{\prime}$ contained inside $\Ecdelta.$
Since the loop $\delta^{\prime}$ passes through the point $q_1,$
the latter in fact belongs to the surface $\Cpzero \subset
\Ecdelta.$ Because $\Pi(\Xdeltahat \times S)=
\overline{\Ecdelta},$ there exists a point $(z_1,p_0) \in
\Xdeltahat \times S$ such that $\Pi(z_1,p_0)=q_1.$

As already discussed in the proof of Corollary \ref{corollary
about peridoic orbits and cycles}, the fact that $H(\deltaprime)
\subset \Cdelta$ is null-homotopic implies that $\deltaprime$
lifts to a loop $\deltaprimehat$ on $\Xdeltahat \times S$ that
passes through the point $(z_1,p_0)$ and its image
$\Pi(\deltaprimehat)=\deltaprime.$ Let
$(z_{j+1},p_0)=\Pmaphatzero^{j}(z_1,p_0)$ for $j=1,...,m-1.$ The
orbit $(z_1,p_0),$...,$(z_m,p_0)$ belongs to $\Ahat \times
\{p_0\}.$ The loop $\deltaprime$ can be regarded as a path from
the point $q_1$ to itself so its lift $\deltahat,$ being also a
loop, is a path from $(z_1,p_0)$ to itself. For that reason, we
can conclude $\Pmaphatezero^m(z_1,p_0)=(z_1,p_0)$ which means that
$(z_1,p_0),...,(z_m,p_0)$ is an $m-$periodic orbit on $\Ahat
\times \{p_0\}.$ Together with that, the orbit is isolated because
the original orbit $q_1,...,q_m$ is isolated.

Let $\eta$ be an embedded in $D_r(0)$ curve, connecting $\e_0$ to
$0$. For convenience, define a natural linear order $\preceq$ on
it so that $0 \prec \e_0.$ By point 2 from Lemma \ref{Lemma
peridoic orbits and families and connection with cycles}, there
exists $D_{r_0}(\e_0) \subset D_r(0)$ for some $r_0>0,$ such that
if $\eta_0=\eta \cap D_{r_0}(\e_0),$ then there is a continuous
family of periodic orbits
$\bigl((z_1(\e),p_0),...,(z_m(\e),p_0)\bigr)_{\e \in \eta_0}$ of
the map $\Pmaphat$ on the cross-section $\Xdeltaprimehat \times
\{p_0\}.$

Define $\etamax \subseteq \eta$ as the maximal relatively open
subset of $\eta$ on which the continuous family
$\bigl((z_1(\e),p_0),...,(z_m(\e),p_0)\bigr)_{\e \in \etamax}$ of
periodic orbits for $\Pmaphat$ exists on $\Xdeltaprimehat \times
\{p_0\}.$ Since $\eta_0 \neq \varnothing$ is a relatively open in
$\eta,$ the inclusion $\eta_0 \subseteq \etamax $ holds and
therefore $\etamax \neq \varnothing.$

By point 3 from Lemma \ref{Lemma peridoic orbits and families and
connection with cycles} there is a continuous family of marked
complex cycles $\{(\Delta_{\e},q_{\e})\}_{\e \in \etamax}$ with
$q_{\e}=\Pi(z_1(\e),p_0).$ Near $\e_0 \in \etamax$ the cycles
$(\Delta_{\e},q_{\e})$ have $\delta_0,m-$fold vertical
representatives $\deltaprimee$ contained in $\Ecdelta$ because for
$\e=\e_0$ the cycle $(\Delta_{\e_0},q_{\e_0})$ has a
$\delta_0,m-$fold vertical representative, namely
$\deltaprime=\deltaprime_{\e_0},$ contained inside the domain
$\Ecdelta.$ We are interested to find out what happens to the
cycles as $\e$ varies on $\etamax.$

Let $\etaprime$ be the set of all $\e$ from $\etamax$ for which
the periodic orbits from the continuous family
$\bigl((z_1(\e),p_0),...,(z_m(\e),p_0)\bigr)_{\e \in \etamax}$ are
contained in $\Ahat \times \{p_0\}.$ As we already saw, at $\e_0$
the orbit $(z_1(\e_0),p_0),...,(z_m(\e_0),p_0)$ is inside $\Ahat
\times \{p_0\}$ and by continuity, the orbits
$(z_1(\e),p_0),...,(z_m(\e),p_0)$ are also contained in $\Ahat
\times \{p_0\}$ for $\e$ near $\e_0.$ This fact shows that
$\etaprime \neq \varnothing$ and in fact it has a nonempty
interior.

Let $\edoublestar=\inf_{\eta}(\etamax)$ be the infimum of
$\etamax$ with respect to the linear ordering on $\eta.$ Then,
$D_{\frac{1}{N}}(\edoublestar) \cap \etamax \neq \varnothing$ for
all $N \in \Naturals.$ Similarly, define
$\estar=\inf_{\eta}(\etaprime)$ as the infimum of $\etaprime.$ The
inclusion $\etaprime \subseteq \etamax$ implies that $\edoublestar
\preceq \estar.$ We are going to show that $\edoublestar \neq
\estar.$

Assume $\edoublestar=\estar,$ that is for all $N \in \Naturals$
there exists $\e_N \in D_{\frac{1}{N}}(\edoublestar) \cap \etamax$
such that $(z_1(\e_N),p_0),...,(z_m(\e_N),p_0)$ is contained in
$\Ahat \times \{p_0\}.$ As explained in point 2 of Lemma
\ref{Lemma peridoic orbits and families and connection with
cycles} the family of periodic orbits
$\bigl((z_1(\e),p_0),...,(z_m(\e),p_0)\bigr)_{\e \in \etamax}$ is
mapped by $\pizero$ to a periodic family
$(x_1(\e),...,x_m(\e))_{\e \in \etamax}$ of the map $\Pmaptilde$
on the surface $\Xdeltaprime.$ Also, the corresponding orbits
$x_1(\e_N),...,x_m(\e_N)$ are inside $Y \subset \Xdeltaprime$ for
$N \in \Naturals.$ In particular, the sequence $\{x_1(\e_N)\}_{N
\in \Naturals}$ is contained in the compact set $Y.$ Then, there
exists $x^{*} \in Y$ and a subsequence $\{x_1(\e_n)\}_{n \in
\Naturals}$ such that $\lim_{n \to \infty} x_1(\e_n)=x_1^*$ and
$\lim_{n \to \infty} \e_n=\edoublestar.$ By continuity, the
identity $\tilde{P}_{\delta_0,\e_n}^m(x_1(\e_n))=x_1(\e_n)$
converges to $\Pmaptildeestar^m(x_1^*)=x_1^*$ as $n \to \infty.$
Generate a periodic orbit $x_1^*,...,x_m^*$ by setting
$x_{j+1}^*=\Pmaptildeestar^j(x_1^*)$ for $j=1,...,m-1.$ Since
$x_{j+1}(\e_n)=\Pmaptildeestar^j(x_1(\e_n))$, the limit for each
$x_j(\e_n)$ is $x_j^*$ as $n \to \infty.$ Thus, the periodic orbit
$x_1^*,...,x_m^*$ is the limit of periodic orbits
$x_1(\e_n),...,x_m(\e_n).$

We will show that under the current assumptions $\edoublestar=0$.
Assume that $\edoublestar \neq 0$. Then $\{\e \in \eta \, : \, \e
\prec \edoublestar\} \neq \varnothing.$ We proceed in a very
similar fashion to that in the proof of Lemma \ref{Lemma peridoic
orbits and families and connection with cycles}. The point $x_1^*
\in Y$ is fixed by the map $\Pmaptildeestar^m.$ Take a complex
chart $(\Uxonestartilde, \tilde{\phi}_{x^{*}_1,\edoublestar})$
form the atlas $\Atlas_{\e}(Y^{\prime})$ around the point $x_1^*$
and a smaller neighborhood $\Uxonestartildeprime \subset
\Uxonestartilde$ of the same point such that
$\Pmaptildeestar^m(\Uxonestartildeprime) \subset \Uxonestartilde.$
Let
$D^{\prime}=\tilde{\phi}_{x^{*}_1,\edoublestar}(\Uxonestartildeprime)
\subset \Disc$ where $\tilde{\phi}_{x^{*}_1,\edoublestar}(x_1^*)=0
\in D^{\prime}.$ Choose $r^*>0$ small enough such that
$$\Pem = \tilde{\phi}_{x^{*}_1,\e} \circ \Pmaptilde^m \circ \tilde{\phi}_{x^{*}_1,\e}^{-1}
\, : \, D^{\prime} \, \longrightarrow \Disc$$ for $\e \in
D_{r^*}(\edoublestar) \subset D_r(0).$ Notice that
$\Pemstar(0)=0.$ The complex valued function
$$\Ftilde : D^{\prime} \to \cc \,\,\, \text{defined as}\,\,\,
\Ftilde(\zeta,\e)=\Pem(\zeta)-\zeta$$ is holomorphic with respect
to $(\zeta,\e) \in D^{\prime} \times D^{r^*}(\edoublestar).$ Since
$\Pemstar(0)=0,$ the point $(0,\edoublestar)$ is a zero of
$\Ftilde,$ that is $\Ftilde(0,\edoublestar)=0.$

We are interested in the zero locus of $\Ftilde$ in $D^{\prime}
\times D^{r^*}(\edoublestar).$ If we assume for a moment that
$\Ftilde(\zeta,\e)\equiv 0$ on $D^{\prime}$ then we would have the
identity $\Pem(\zeta)\equiv \zeta$ on $D^{\prime}$ and therefore
$\Pmaptilde^m(x)\equiv x$ on the open subset
$\Uxonestartilde^{\prime} \subset \Xdeltaprime.$ Because of the
analyticity of $\Pmaptilde^m(x)$ with respect to both $x$ and
$\e,$ the identity $\Pmaptilde^m(x)\equiv x$ will hold on all of
$\Xdeltaprime$ and for all $\e \in D_{r}(0)$. In particular, it
will be true for $\e=\e_0$. But for that value the map
$\Pmaptilde^m$ has an isolated fixed point $x_1(\e_0) \in Y
\subset \Xdelta$ which leads to a contradiction. Therefore
$\Ftilde$ is not identically zero.

There are two cases for $\Ftilde.$ Either
$\Ftilde(\zeta,\edoublestar) \equiv 0$ or
$\Ftilde(\zeta,\edoublestar) \not\equiv 0$ for $\zeta \in
D^{\prime}.$ For both of those options $\Ftilde$ can be written as
$$\Ftilde(\zeta,\e)=(\e-\edoublestar)^bF(\zeta,\e)$$
where $F(\zeta,\edoublestar)\not\equiv 0$ and $b \geq 0.$ When $b
> 0$ we have the first case and when $b=0$ we have the second
case.

Let us look at the zero locus of $F.$ By Weierstrass Preparation
Theorem \cite{Ch}, \cite{GR}, $F$ can be written as
$$F(\zeta,\e)=\prod_{j=1}^{s}(\zeta - \alpha_j(\e))\theta(\zeta,\e),$$
where $\theta(0,\edoublestar) \neq 0$ and $\{\alpha_j(\e) \, : \,
j=1,...,s\}$ depend analytically on $\e,$ satisfying the
equalities $\alpha_1(\eprime)=...=\alpha_s(\eprime)=0$ and
possibly branching into each other. Without loss of generality, we
can think that $D^{\prime}$ is chosen small enough so that
$\nu(\zeta,\e) \neq 0$ for all $(\zeta,\e) \in D^{\prime}\times
D_{r^*}(\edoublestar).$ Let
$\tilde{\alpha}_j(\e)=\phixonestaretilde^{-1}(\alpha_j(\e))$.
Since $x_1(\e_n) \to x_1^*,$ there exists $N_0 \in \Naturals$ such
that $x_1(\e_n) \in \Uxonestartilde^{\prime}$ for $n>N_0.$ By the
continuity of $x_1(\e),$ for each $\e \in D_{r^*}(\edoublestar)
\cap \etamax$ we have that $x_1(\e)=\tilde{\alpha}_j(\e)$ for some
$j=1,..,m.$ Thus, $x_1(\e)$ converges to $x_1^*$ as $\e \to
\edoublestar$ always staying on the zero locus of $F.$ Thus we can
extend $x_1(\e)$ continuously on $\eta$ past $\edoublestar$ by
setting $x_1(\e)=\tilde{\alpha}_j(\e)$ for $\e \in
D_{r^*}(\edoublestar) \cap \{\e \in \eta \, : \, \eta \preceq
\edoublestar\}.$ By construction, the identity
$\Pmaptilde^m(\tilde{\alpha}_1(\e))= \tilde{\alpha}_1(\e)$ holds
and if we set $x_{j+1}(\e)=\Pmaptilde^j(\tilde{\alpha}_1(\e))$ we
obtain a continuation of the family $x_1(\e),...,x_m(\e)$ on the
relatively open arc $D_{r^*}(\edoublestar) \cap \{\e \in \eta \, :
\, \eta \preceq \edoublestar\}$. As a result we have a continuous
family $\bigl(x_1(\e),...,x_m(\e)\bigr)_{\e \in \etatilde}$ of
periodic orbits for $\Pmaptilde$ where $\etatilde =
(D_{r^*}(\edoublestar) \cap \{\e \in \eta \, : \, \eta \preceq
\edoublestar\}) \cup \etamax$ is relatively open in $\eta$.

Since the family $(z_1(\e),p_0),...,(z_m(\e),p_0)$ is the lift of
$x_1(\e),...,x_m(\e)$ for $\e \in \etamax$ and the latter extends
on $\etatilde \supset \etamax,$ the former also extends on
$\etatilde$ as a family of periodic orbits for $\Pmaphat$ on the
cross-section $\Xdeltaprimehat \times \{p_0\}.$ This conclusion
contradicts the maximality of $\etamax$, stemming from the
assumption that $\edoublestar \neq 0.$ Therefore $\edoublestar=0$
and $x_1(0),...,x_m(0)$ is a periodic orbit of
$\tilde{P}_{\delta_0,0}=id_{\Xdeltaprime}.$ For that reason,
$x_1(0)=...=x_m(0)=x^*$ inside $\Xdeltaprime.$

Take a complex chart $(\Uxstartilde,\tilde{\phi}_{x^*,0})$ around
the point $x^*$ and choose a smaller neighborhood
$\Uxstartilde^{\prime} \subset \Uxstartilde$ of $x^*$ such that
$\Pmaptilde^k(\Uxstartilde^{\prime}) \subset \Uxstartilde$ for all
$k=1,...,m$ and $\e \in D_{r_0}(0),$ where $r_0>$ is small enough.
Let $D^{\prime}=\tilde{\phi}_{x^*,0}(\Uxstartilde^{\prime})
\subset \Disc$ and
$$\Pe=\phixstaretilde \circ \Pmaptilde \circ \phixstaretilde^{-1} \, : \, D^{\prime} \longrightarrow \Disc.$$
Denote by $\zeta_j(\e)=\phixstaretilde(x_j(\e))$ for $\e \in
D_{r_0}(0) \cap \etamax = \eta_0$ and $j=1,...,m$. Then
$\zeta_1(\e),...,\zeta_m(\e)$ is a periodic orbit for $\Pe$ in
$D^{\prime}$. Notice,that due to the holomorphic nature of the map
$\Pmaptilde,$ those $\e \in \etamax$ for which $x_i(\e)=x_j(\e),$
where $1\leq i < j \leq m,$ are isolated because the family at
$\e_0$ consists of an $m-$periodic point. As before $\Pe(\zeta)$
is holomorphic with respect to $(\zeta,\e)$. Then we can write the
map as $$\Pe(\zeta)=\zeta + \e^l I(\zeta) + \e^{l+1}R(\zeta,\e)$$
where $I(\zeta) \not\equiv 0$ and $l \geq 1.$ If we iterate the
map  $m$ times we obtain the representation
$$\Pe^m(\zeta)=\zeta + \e^l m I(\zeta) + \e^{l+1}R_{(m)}(\zeta,\e).$$

For $\e \in \eta_0 - \{0\}$ the equations
\begin{align*}
\Pe(\zeta)-\zeta &= \e^l(I(\zeta) + \e R(\zeta,\e))=0 \,\,\,\,
\text{and}\\
\Pe^m(\zeta)-\zeta &= \e^l(m I(\zeta) + \e R_{(m)}(\zeta,\e))=0
\end{align*}
are divisible by $\e^l$ and thus, become
\begin{equation}\label{equation for periodic orbits}
I(\zeta) + \e R(\zeta,\e)=0 \,\,\,\,\, \text{and} \,\,\,\,\, m
I(\zeta) + \e R_{(m)}(\zeta,\e)=0
\end{equation}
The function $I(\zeta)$ is not identically zero, so it has
isolated zeroes. Choose $D^{\prime \prime} \subset D^{\prime}$ to
be a small closed disc centered at zero, so that no zeroes of
$I(\zeta)$ are contained in $D^{\prime \prime} - \{0\}.$ In
particular, $I(\zeta) \neq 0$ for $\zeta \in
\partial D^{\prime \prime}.$ We can decrease the parameter radius
$r_0>0$ enough so that by Rouche's Theorem \cite{Smth} the
equations (\ref{equation for periodic orbits}) will have the same
number of zeroes, counting multiplicities, as the equation
$I(\zeta)=0.$ Clearly, all zeroes of $\Pe(\zeta)-\zeta$ are zeroes
of $\Pe^m(\zeta)-\zeta$ because the fixed points of $\Pe$ are
fixed points of $\Pe^m$ but not the other way around. On the other
hand, as already noted, for almost every $\e \in D_{r_0}(0)$ there
is an $m-$periodic orbit $\zeta_1(\e),...,\zeta_m(\e)$ for the map
$\Pe$ inside $D^{\prime \prime}.$ Thus, we can see that
$\Pe^m(\zeta)-\zeta$ has at least $m$ zeroes more than
$\Pe(\zeta)-\zeta,$ which contradicts the fact that both of these
should have the same number of zeroes. The contradiction comes
from the assumption that $\edoublestar=\estar.$ Therefore we can
conclude that $\edoublestar \neq \estar$ and in fact $\edoublestar
\prec \estar.$

Let $\eta_1 = \{\e \in \etamax \, : \, \edoublestar \prec \e \prec
\estar\}.$ Then for any $\e_1 \in \eta_1$ at least one
$(z_{j_0}(\e_1),p_0)$ is contained in $\Xdeltaprimehat \times
\{p_0\}$ but not in $\Ahat \times \{p_0\}.$ It follows form here
that $(z_1(\e_1),p_0)$ is not contained in $\Xdeltahat \times
\{p_0\},$ otherwise if $(z_1(\e_1),p_0)$ were in $\Xdeltahat
\times \{p_0\},$ then
$(z_{j_0}(\e_1),p_0)=P_{\delta_0,\e_1}^{j_0-1}(z_1(\e_1),p_0)$
would be inside $\Ahat \times \{p_0\},$ which is not the case.

By point 3 of Lemma \ref{Lemma peridoic orbits and families and
connection with cycles} there exists a continuous family  of
marked cycles $\{(\Delta_{\e},q_{\e})\}_{\e \in \etamax}$, where
$q_{\e}=\Pi(z_1(\e),p_0)$. For any $\e_1 \in \eta_1 \subset
\etamax$ there are two options. The first one is that $q_{\e_1}
\in \Cpzeroprime - \Cpzero$. Then, no representative of
$(\Delta_{\e_1},q_{\e_1})$ is contained in $\Ecdelta$ because all
of them pass through $q_{\e_1}$ and $q_{\e_1}$ is not in
$\Ecdelta$. The second option is that $(z_1(\e_1),p_0)$ belongs to
$\Pi^{-1}(\Cpzero) = \cup_{\gamma \in \Gamma} \,
\bigl(\gamma(\Cdeltahat) \times \{p_0\}\bigr)$ but does not belong
to $\Cdeltahat \times \{p_0\}.$ In this case, there exists $\gamma
\in \Gamma - \Gamma_0$ such that $(z_1(\e_1),p_0) \in
\gamma(\Cdeltahat) \times \{p_0\}.$ By point 2 of Lemma
\ref{Lemma_periodic_orbits_cycles_lifted}, any representative
$\deltaprime_{\e_1}$ of the marked complex cycle
$(\Delta_{\e_1},q_{\e_1})$, that is contained in $\Edelta$, is not
$\delta_0,m-$fold vertical. Thus, Theorem \ref{rapid-evolution} is
true with $\sigma = \etamax$. $\square$ \vspace{2mm}

\section{Foliations with Multi-Fold Limit Cycles}

In this chapter we discuss an example, such that for any $m \in
\Naturals,$ a family of polynomial foliations of the form
\ref{basic} has a limit $m$-fold vertical cycle. More specifically
we are going to look at the two-parameter family
\ref{twoparameter} already introduced in Section \ref{Section Main
Thms}.

\subsection{The Foliation and Its Poincar\'e Map}
As defined earlier, the foliation $\Foliation^{a,\e}$ is given by
the complex line field
\begin{equation} \label{equation two parameter family}
F^{a,\e} = \ker\Bigl( d H + \e \bigl((H-1) (y dx - x dy) + a y \,
d H \bigr)\Bigr),
\end{equation}
with a transverse to infinity integrable part $H = x^2 + y^2$ and
parameters $\e$ and $a$. The leaf
$$S_1=\{(x,y) \in \CC \, | \, x^2 + y^2 =1\}$$ tangent to $\ker(dH)$
is diffeomoprhic to a cylinder with a nontrivial loop on it
$\delta_0 = S_1 \cap \RR$. It is very important to point out that,
in fact, $S_1$ is tangent to the line field $F^{a,\e}$ and
therefore is a leaf of the foliation $\Foliation^{a,\e}$ for all
$(a,\e) \in \cc^* \times \cc^*.$

Define $A(\delta_0)$ as a tubular neighborhood of $\delta_0$ on
the surface $S_1$ and $N(\delta_0)$ as a tubular neighborhood of
$A(\delta_0)$ in $\CC.$ Let $$\Band_{r_0} =\{\zeta \in \cc \, : \,
|\text{Im}(\zeta)| < r_0 \}$$ be a an infinite horizontal band in
$\cc$ of width $r_0$ and let $$D_{r_0}(1) = \{\xi \in \cc \, : \,
|\xi - 1| \leq r_0\}$$ be the disc of radius $r_0$ centered at
$1$. Consider the map
$$f_1 : \Band_{r_0} \times D_{r_0}(1) \to N(\delta_0) \,\,\, \text{defined by}
\,\,\, f_1 : (\zeta,\xi) \mapsto (\xi \cos{\zeta}, \xi
\sin{\zeta}).$$ Without loss of generality, we can think that
$f_1(\Band_{r_0} \times D_{r_0}(1)) = N(\delta_0).$ In other
words, $f_1$ can be thought of as the universal covering map of
$N(\delta_0)$. Notice, that implies $f_1(\Band_{r_0} \times \{1\})
= A(\delta_0) \subset S_1$.

The pull-back $f_1^*F^{a,\e}$ on $\Band_{r_0} \times D_{r_0}(1)$
of the line field $F^{a,\e}$ is
$$f_1^*F^{a,\e} = \ker{\bigl(d(\xi^2) - \e (\xi^2 - 1)\xi^2 \, d\zeta + a \e \, \xi \sin{\zeta} \, d(\xi^2)\bigr)}.$$
For $0<r_1<1,$ define the map $$f_2 : \Band_{r_0} \times
D_{r_1}(0) \to \Band_{r_0} \times D_{r_0}(1)\,\,\, \text{where}
\,\,\, f_2 : (z,w) \mapsto \Bigl(z, \frac{1}{\sqrt{1 -
w}}\Bigr).$$ Composing the maps $f_1$ and $f_2$ we obtain $$f =
f_1 \circ f_2 \, : \, \Band_{r_0} \times D_{r_1}(0) \,
\longrightarrow \, N(\delta_0).$$ Then the pull-back $f^*F^{a,\e}$
is
$$f^*F^{a,\e} = \ker\left(\frac{1}{(1-w)^2}\Bigl(dw - \e\,w dz + \e a \, \frac{\sin{z}}{\sqrt{1-w}}\, dw
\Bigr)\right)$$ and since $\frac{1}{(1-w)^2}$ is well defined and
nonzero for $w \in D_{r_1}(0),$ the line field becomes
$$f^*F^{a,\e} = \ker{\Bigl(dw - \e\,w dz + \e a \, \frac{\sin{z}}{\sqrt{1-w}}\, dw
\Bigr)}.$$ The holomorphic function $\mu_{\e}(z)=e^{-\e z}$ is
nonzero everywhere, so
\begin{align*}
f^*F^{a,\e} &= \ker{\Bigl(e^{-\e z} dw - \e\,w e^{-\e z}\,dz + \e
a \, \frac{e^{-\e z} \sin{z}}{\sqrt{1-w}}\, dw \Bigr)} \\
&= \ker{\Bigl(d(w e^{-\e z}) + \e a \, \frac{e^{-\e z}
\sin{z}}{\sqrt{1-w}}\, dw \Bigr)} \\
&= \ker{(d J^{(\e)} + a \omega^{(\e)})}
\end{align*}
$$\text{where} \,\,\,\, J^{(\e)} = w e^{-\e z} \,\,\,\, \text{and}\,\,\,\, \omega^{(\e)} =
\frac{e^{-\e z} \sin{z}}{\sqrt{1-w}}\, dw.$$

Our next step is to define the Poincar\'e transformation for the
foliation $\Foliation^{a,\e},$ using the local chart $f$ on the
tubular neighborhood $N(\delta_0)$ of the loop $\delta_0.$ Denote
the desired map by
$$P_{a,\e} \, : \, D_{r_1}(0) \, \longrightarrow \, \cc.$$ We are
going to explain how it is constructed.

Define the path $\deltahat_0 = \{(t,0) \in \Band_{r_0} \times
\{0\} \, : \, t \in [0,2\pi]\}.$ Then $f(\deltahat_0)=\delta_0.$
The segment $\deltahat_0$ can be lifted to a path
$\delta_{a,\e}(u)$ on the leaf of $\Foliation^{a,\e}$ passing
through the point $(0,u) \in \{0\}\times D_{r_1}(0),$ so that if
$pr_1 : (z,w) \mapsto z$ then
$pr_1(\delta_{a,\e}(u))=\deltahat_0.$ The lift $\delta_{a,\e}(u)$
has two endpoints. The first one is $(0,u)$ and the second one we
denote by $(2\pi,P_{a,\e}(u))$. When a=0, the map $P_{0,\e}(u)$
comes from the foliation $\Foliation^{0,\e}$ which in our tubular
neighborhood is given by $\ker(d(we^{-\e z})).$ Then,
$\delta_{0,\e} = \{(t,u e^{\e t}) \, : \, t \in [0,2\pi]\}$ and so
$P_{0,\e}=e^{2 \pi \e}u.$ Since $\deltahat_{a,\e}(0)=\deltahat_0$,
the equality $P_{a,\e}(0)=0$ holds for all $(a,\e)$. As a result,
the Poincar\'e transformation can be written down as
$$P_{a,\e}(u)=e^{2 \pi \e}u + a I(u,\e)u + a^2 G(u,a,\e)u$$ and
its $k$-th iteration can be expressed as
$$P^k_{a,\e}(u)=e^{2 k \pi \e}u + a I_{(k)}(u,\e)u + a^2 G_{(k)}(u,a,\e)u.$$
If $m=\frac{i}{m}$ then after $m$ iterations the map becomes
$$P^m_{a,\frac{i}{m}}(u)=u + a I_{(m)}\Bigl(u,\frac{i}{m}\Bigr)u + a^2 G_{(m)}\Bigl(u,a,\frac{i}{m}\Bigr)u.$$
Notice that in this case, by lifting $\deltahat_0^m$ we obtain the
path \begin{equation} \label{curve formula}
\delta^{(m)}_{a,\frac{i}{m}}(u)=\{(t,e^{\frac{i}{m}t}u)\, | \, t
\in [0,2 \pi m]\} \end{equation} with endpoints $(0,u)$ and $(2
\pi m, u)$.

In order to study the periodic orbits of $P_{a,\e}(u)$, we are
going to look at the difference $P^m_{a,\frac{i}{m}}(u)-u$. Since
$\bigl(d J^{({i}/{m})} +
\omega^{({i}/{m})}\bigl)|_{\delta_{a,{i}/{m}}(u)} = 0,$ it can be
concluded that
\begin{align*}
&\int_{\delta_{a,{i}/{m}}(u)} \bigl(d J^{({i}/{m})} + a
\omega^{({i}/{m})}\bigr)=0 \,\,\,\, \text{and hence} \\
&\int_{\delta_{a,{i}/{m}}(u)} d J^{({i}/{m})} = - a
\int_{\delta_{a,{i}/{m}}(u)}\omega^{({i}/{m})}.
\end{align*}
The one-form $d J^{({i}/{m})}$ is exact and yields
\begin{align} \label{Formula Integral Equalities}
\nonumber P^m_{a,i/m}(u)-u &= P^m_{a,i/m}(u)e^{-2 \pi}-ue^{0}\\
\nonumber &= J^{({i}/{m})}(2 \pi m,u) - J^{({i}/{m})}(0,u) \\
          &= \int_{\delta_{a,{i}/{m}}(u)} d J^{({i}/{m})} \\
\nonumber &= - a \int_{\delta_{a,{i}/{m}}(u)}\omega^{({i}/{m})}.
\end{align}
Dividing equation (\ref{Formula Integral Equalities}) by $a$ and
taking into account that the limit of the left hand side is
$I_{(m)}(u,i/m)u,$ as well as $\delta_{a,{i}/{m}}(u) \to
\delta_{0,{i}/{m}}(u),$ when $a \to \infty,$ we can conclude that
$$I_{(m)}(u,i/m)u = - \int_{\delta_{0,{i}/{m}}(u)}\omega^{({i}/{m})}.$$
Now, remembering that $\delta_{0,{i}/{m}}(u)$ is of the form
(\ref{curve formula}) compute
\begin{eqnarray*} \label{}
I_{(m)}(u,i/m)u &=& - \int_{\delta_{0,{i}/{m}}(u)}
\frac{e^{-\frac{i}{m}z}
\sin{z}}{\sqrt{1-w}}\, d w \\
&=& - \int_{0}^{2 \pi m} \frac{e^{-\frac{i}{m}t}
\sin{t}}{\sqrt{1-u e^{\frac{i}{m}t}}}\,\Bigl( \frac{i}{m} u e ^{\frac{i}{m}t} \Bigr) dt \\
&=& - \frac{i u}{m} \int_{0}^{2 \pi m} \frac{\sin{t}}{\sqrt{1-u
e^{\frac{i}{m}t}}} \, dt.
\end{eqnarray*}
Since both sides of the equation are divisible by $u,$
\begin{equation} \label{first approximation formula one}
I_{(m)}(u,i/m) = - \frac{i}{m} \int_{0}^{2 \pi m}
\frac{\sin{t}}{\sqrt{1-u e^{\frac{i}{m}t}}} \, dt
\end{equation}
To solve the integral, notice that $1/\sqrt{1 - w}$ is well
defined and holomorphic in the disc $D_{r_1}(0) \not\ni 1$ so it
expands as convergent series
$$(1 - w)^{-\frac{1}{2}} = \sum_{k=0}^{\infty} b_k w^k ,$$ where $b_k
= (-1)^k \frac{-\frac{1}{2}\bigl(-\frac{1}{2} - 1
\bigr)\bigl(-\frac{1}{2} - 2\bigr)...\bigl(-\frac{1}{2} -
(k-1)\bigr)}{k !} \neq 0.$ Thus,
\begin{eqnarray} \label{eqn 1}
\nonumber \int_{0}^{2 \pi m} \frac{\sin{t}}{\sqrt{1-u
e^{\frac{i}{m}t}}} \, dt &=& \int_{0}^{2 \pi m} \left(
\sum_{k=0}^{\infty} b_k e^{i
\frac{k}{m} t} u^k \right) \sin{t} \, dt \\
&=& \sum_{k=0}^{\infty} b_k \left(\int_{0}^{2 \pi m} e^{i
\frac{k}{m} t} \sin{t} \, dt \right) u^k.
\end{eqnarray}
The value of the integral depends on the coefficients of (\ref{eqn
1}) that depend on the integral
\begin{eqnarray*}
\int_{0}^{2 \pi m} e^{i \frac{k}{m} t} \sin{t} \, dt &=&
\frac{1}{2i} \int_{0}^{2 \pi m} e^{i \frac{k}{m} t}(e^{i t} - e^{-
i t}) \, dt \\
&=& \frac{1}{2i} \int_{0}^{2 \pi m} \bigl(e^{i\frac{k + m}{m} t} -
e^{i\frac{k - m}{m} t} \bigr) \, dt
\end{eqnarray*}
When $k \neq m$ the primitive of the function $\bigl( e^{i\frac{k
+ m}{m} t} - e^{i\frac{k - m}{m} t} \bigr)$ under the integral is
again $2 \pi m$-periodic, leading to the conclusion that the
integral is zero. When $k=m$ the integral becomes
\begin{eqnarray*}
\int_{0}^{2 \pi m} e^{i t} \sin{t} \, dt &=& \frac{1}{2i}
\int_{0}^{2 \pi m} e^{i t}(e^{i t} - e^{-
i t}) \, dt \\
&=& \frac{1}{2i} \int_{0}^{2 \pi m} \bigl(e^{i 2 t} - 1 \bigr) \,
dt \\
&=& \frac{1}{(2i)^2} \Bigl( e^{2 i t} \Bigr)_{0}^{2 \pi m} -
\frac{\pi m}{i} \\
&=& i \pi m
\end{eqnarray*}
The computations above lead to
\begin{align*}
I_{(m)}\bigl(u,\frac{i}{m}\bigr) = - \frac{i}{m} \, b_m \, i \pi m
\,\, u^m = \pi b_m \, u^m.
\end{align*}
Finally, for $\e=\frac{i}{m}$, setting $c = \pi b_m \neq 0,$ the
Poincar\'e map takes the form
\begin{equation} \label{Equation Pmap with computed first melnikov function}
P_{a,\frac{i}{m}}^m(u) = u + a\, c u^{m+1} + a^2
G_{(m)}\bigl(u,a,{i}/{m}\bigr)u.
\end{equation}

\subsection{Existence of Periodic Orbits and Multi-Fold Cycles}

This section establishes the result of Theorem \ref{example}.
\vspace{2mm}
\paragraph{\em Proof of Theorem \ref{example}.} From the discussion in the
introduction, the existence of a multi-fold limit cycle of
$\Foliation^{a,\e}$ follows from the existence of an isolated
$m$-periodic orbit of the Poincar\,e transformation
$P_{a,\e}.$ %It follows from the construction of the map that a
representative of the cycle in this case will be contained in the
the tubular neighborhood $N(\delta_0)$ and therefore free
homotopic to $\delta_0^m$ in it. This means the limit cycle will
be $\delta_0,m$-fold. Thus, the main objective  will be to show
that $P_{a,\e}$ has an isolated $m$-periodic orbit.

Assume we can show that the periodic orbit exists. After fixing
the appropriate $a,$ so that the presence of the periodic orbit is
secured, Theorem \ref{thm-pmap} will apply to the family
$\Foliation^{a,\e}$ and by picking $p_0 = (1,0),$ we can construct
a global smooth cross-section $\Bpzero$ diffeomorphic to the
punctured plain $B = \cc^{*}$. In fact, the topology of the
integrable leaves is so simple (they are cylinders) that $\Bdelta
= B$ and so $\Edelta=E.$ The regions $\Cdeltaprime, \Cdelta$ and
$\Aprime$ will be nested annuli of very large width and we will
have, as Theorem \ref{thm-pmap} implies, a global Poincar\'e
transformation on a cross-section $\Cpzeroprime \subset \Bpzero.$
It is easy to notice that, as Lemma \ref{Lemma Complex Pmap}
reveals, the map $P_{a,\e}$ can be regarded simply as a
representation of the $\Pmap$ in one of the complex charts
introduced in Lemma \ref{lemma complex atlas}. Theorem
\ref{thm-pmap} shows, that the complex cycle corresponding to the
$m$-periodic orbit of $P_{a,\e}$ will be in fact limit
$\delta_0,m$-fold vertical and will satisfy the premises of
Theorem \ref{rapid-evolution}. Thus, the limit multi-fold vertical
cycle of $\Foliation^{a,\e}$ will be subject to rapid evolution as
described in Theorem \ref{rapid-evolution}.

In the context of the preceding two paragraphs, a small remark is
in order. The theory, developed in the sections preceding the
current one, has to undergo a small correction. Originally, our
assumption was that $B$ is a hyperbolic Riemann surface covered by
the disc $\Disc$. In our example, $B$ is in fact non-hyperbolic
and is covered by $\cc.$ Since $\cc$ is still contractible, all
the proofs and construction will be essentially the same and the
correction will be merely a matter of change in some notations.

We have the radii $r_1 > 0, r_2>0$ and $\bar{r}_3>0$ so that for
any $(a,\e) \in D_{r_2}(0) \times D_{\bar{r}_3}$ the map
$P_{a,\e}\, : \, D_{r_1}(0) \longrightarrow \cc$ is well defined.
Let $m>0$ be such that $i/m \in D_{\bar{r}_3}(0).$

\begin{lem} \label{Lemma Example Periodic Orbit Existence}
There exist $\e_m$ near $\frac{i}{m}$ and a parameter $a_m$ such
that for all $\e$ in a neighborhood of $\e_m,$ the map $P_{a,\e}$
has an isolated periodic orbit of period $m$.
\end{lem}

\begin{proof}
The verification of the claim depends on four facts. Putting them
together will help us determine the values of the parameters $a$
and $\e.$ As before, in order to find a periodic orbit for the map
$P_{a,\e}(u),$ we are going to look at the equation
\begin{equation}\label{general equation}
P_{a,\e}^m(u) - u=0.
\end{equation} Whenever $a \neq 0$ we can rewrite (\ref{general equation}) in the form
\begin{equation*}
\frac{e^{2 \pi m \e} - 1}{a} \,\, u +  I_{(m)}(u,\e) u + a \,
G_{(m)}(u,a,\e) u = 0.
\end{equation*}
Furthermore, having in mind that $u=0$ is always a solution of
(\ref{general equation}), we can divide by $u$ and obtain
\begin{equation} \label{Eq m a epsilon}
g(u,a,\e) = \frac{e^{2 \pi m \e}-1}{a} + I_{(m)}(u,\e) + a \,
G_{(m)}(u,a,\e)=0
\end{equation}
for $u \in D_{r_1}(0), a \in D_{r_2}(0) - \{0\}$ and $\e \in
D_{\bar{r}_3}(0).$ \vspace{2mm}

\paragraph{\bf Fact 1.} Let us focus on the equation
\begin{equation} \label{Eq m a i/m}
g\Bigl(u,a,\frac{i}{m}\Bigr) = I_{(m)}\Bigl(u,\frac{i}{m}\Bigr) +
a G_{(m)}\Bigl(u,a,\frac{i}{m}\Bigr)=0
\end{equation}
If necessary, decrease the radius $r_2>0$ enough so that if we set
\begin{equation*}
\mathcal{M}(r_1,r_2) = \max{\left\{|a|
\left|G_{(m)}\Bigl(u,a,\frac{i}{m}\Bigr)\right| \, \, : \, \,
|u|=r_1 \, \, \text{and} \, \, a \in D_{r_2}(0)\right\}}
\end{equation*}
then $\mathcal{M}(r_1,r_2) < |c| \, r_1^m.$ Since
$I_{(m)}\Bigl(u,\frac{i}{m}\Bigr) = c u^m$, it follows that for
$|u|=r_1$ and for any $a \in D_{r_2}(0)$
$$|c|\, |u|^m = |c| \, r_1^m > \mathcal{M}(r_1,r_2) \geq  |a|
\left|G_{(m)}\Bigl(u,a,\frac{i}{m}\Bigr)\right|,$$
 so by Rouche's Theorem \cite{Smth}, equation (\ref{Eq m a i/m}) has
 exactly $k$ zeroes $u_1(a)$, $u_2(a),$...,$u_m(a)$ in $D_{r_1}(0)$, counted with
 multiplicities. \vspace{2mm}

 \paragraph{\bf Fact 2.} Let $\mu(\e) = \min{ \{ |e^{2 \pi k \e} - 1| \,\, : \,\, 1 \leq k \leq m-1
 \}}$. Regarded as a function, $\mu(\e)$ is continuous and
 $\mu(i/m) > 0.$ Hence, there exists $r_3>0$, such that
 $\overline{D_{r_3}(i/m)} \subset D_{\bar{r}_3}(0)$. Moreover, there exists a constant $\mu > 0,$ such that
 $\mu(\e)>\mu$ for any $\e \in D_{r_3}(i/m)$. If needed,
 decrease $r_2>0$ so that
\begin{equation*} \label{Fact 2 inequality}
 \max{\bigl\{ \, |a| \, \bigl|I_{(k)}(u,\e) + a G_{(k)}(u,a,\e)\bigr| \,\, : \,\, 1 \leq k \leq m-1 \bigl\}}  <  \mu
\end{equation*}
 for all $u \in D_{r_1}(0), a \in D_{r_2}(0)$ and $\e \in
 D_{r_3}(i/m)$. \vspace{2mm}

\paragraph{\bf Fact 3.} Equation (\ref{Eq m a epsilon}) can take
the form
\begin{eqnarray} \label{Eq m a epsilon rewritten}
g(u,a,\e) = g\Bigl(u,a,\frac{i}{m}\Bigr) + \Bigl(g(u,a,\e) -
g\Bigl(u,a,\frac{i}{m}\Bigr)\Bigr) = 0
%\nonumber 0 &=& \frac{e^{2 \pi m \e}-1}{a} + I_{(m)}(u,\e) + a \,
%G_{(m)}(u,a,\e) \\
% \nonumber &=& I_{(m)}\Bigl(u,\frac{i}{m}\Bigr) + a \, G_{(m)}\Bigl(u,a,\frac{i}{m}\Bigr) + \left( I_{(m)}(u,\e) -
% I_{(m)}\Bigl(u,\frac{i}{m}\Bigr)\right) + \\ &+& a \left( G_{(m)}(u,a,\e) -
% G_{(m)}\Bigl(u,a,\frac{i}{m}\Bigr)\right) + \frac{e^{1 \pi m \e}-1}{a}
\end{eqnarray}
For some specific $a \in D_{r_2}(0)-\{0\}$, Fact 1 reveals that
whenever $|u|=r_1,$ the following inequalities hold:
\begin{equation*}
\begin{split}
%\left|I_{(m)}\Bigl(u,\frac{i}{m}\Bigr) + a \,
%G_{(m)}\Bigl(u,a,\frac{i}{m}\Bigr)\right|
\left|g\Bigl(u,a,\frac{i}{m}\Bigr)\right| \geq
\left|I_{(m)}\Bigl(u,\frac{i}{m}\Bigr)\right| - |a| \,
\left|G_{(m)}\Bigl(u,a,\frac{i}{m}\Bigr)\right| > 0.
\end{split}
\end{equation*}
Hence, $\mu_1(a)= \min{\left\{
\left|g\Bigl(u,a,\frac{i}{m}\Bigr)\right| \, : \, |u|=r_1
\right\}}
> 0$
Notice, that for any nonzero $a \in D_{r_2}(0)$ one can find a
radius $r_3(a) > 0,$ continuously depending on $a$, such that
\begin{align*} \label{Large inequatity fact 3}
\max{\left\{\left| g(u,a,\e) - g\Bigl(u,a,\frac{i}{m}\Bigr)\right|
\, : \,  |u|=r_1, \,\,\, \e \in
D_{r_3(a)}\bigl(i/m\bigr)\right\}}< \mu_1(a),
\end{align*}
Because of the last inequality, it follows by Rouche's Theorem
that equation (\ref{Eq m a epsilon}) has as many solutions as
equation (\ref{Eq m a i/m}). Thus, due to Fact 1, (\ref{Eq m a
epsilon}) has exactly $m$ solutions $u_1(a,\e),...,u_m(a,\e)$,
counted with multiplicities. If we set $$W = \bigsqcup\limits_{0
\neq a \in D_{r_2}(0)} \,\Bigl( \{a\} \times
D_{r_3(a)}(i/m)\Bigr),$$ then $W$ is open and $\overline{W} \ni
(0,\frac{i}{m}).$\vspace{2mm}

\paragraph{\bf Fact 4.} Let $g_0(a,\e) = (e^{2 \pi m \e} - 1) + a
\, I_{(m)}(0,\e) + a^2 \, G_{(m)}(0,a,\e)$. Notice, that
$g_0\bigl(0, \frac{i}{m}\bigr) = 0 \,\,\, \text{and} \,\,\,
\frac{\partial g_0}{\partial \e}\bigl(0, \frac{i}{m}\bigr) = 2 \pi
m \neq 0.$ Hence, by the inverse function theorem, it follows that
for possibly decreased $r_2>0$ there exists a holomotphic function
$\chi : D_{r_2}(0) \to D_{r_3}\bigl(\frac{i}{m}\bigr)$ such that
$\chi(0) = \frac{i}{m}$ and $g_0(a,\chi(a))=0$ for all $a \in
D_{r_2}(0).$ From here, we can see that the zero locus of $g_0$
inside the product domain $D_{r_2}(0) \times
D_{r_3}\bigl(\frac{i}{m}\bigr)$ is
$$Z = \{(a,\e) \,\, : \,\, g_0(a,\e)=0\} =
\{(a,\chi(a)) \,\, : \,\, a \in D_{r_2}(0)\}.$$ The set $Z$ is
relatively closed in $D_{r_2}(0) \times
D_{r_3}\bigl(\frac{i}{m}\bigr)$ so its complement
$\bigl(D_{r_2}(0) \times D_{r_3}\bigl(\frac{i}{m}\bigr)\bigl) - Z$
is open and nonempty. Therefore, $W \cap \Bigl[\bigl(D_{r_2}(0)
\times D_{r_3}\bigl(\frac{i}{m}\bigr)\bigl) - Z \Bigr] \neq
\varnothing$ is open as well. \vspace{2mm}

Now we are ready to complete the proof of the lemma. Let
$(a_m,\e_m) \in W \cap \Bigl[\bigl(D_{r_2}(0) \times
D_{r_3}\bigl(\frac{i}{m}\bigr)\bigl) - Z \Bigr]$. Apply the
results from Fact 4 to obtain
\begin{align*}
  g_0(a_m,\e_m) = (e^{2 \pi m \e_m} - 1) &+ a_m \, I_{(m)}(0,\e_m) +\\
                                         &+ a^2_m \, G_{(m)}(0,a_m,\e_m) \neq 0.
\end{align*}
Hence, the equation
\begin{align*}
  P_{\e_m,a_m}^m(u) - u = (e^{2\pi m \e_m} - 1) \, u
  &+ a_m \,I_{(m)}(u,\e_m)u +\\
  &+ a^2_m \, G_{(m)}(u,a_m,\e_m)u = 0
\end{align*}
has $u_0=0$ as a simple root.

Since $(a_m,\e_m) \in W$, it follows from Fact 3 that whenever
$|u|=r_1$ the following inequality holds
\begin{align*}
%I_{(m)}\Bigl(u,\frac{i}{m}\Bigr) + a_m \, G_{(m)}\Bigl(u,a_m,\frac{i}{m}\Bigr)
\left|g\Bigl(u,a_m,\frac{i}{m}\Bigr)\right| \geq \mu_1(a_m)
> \left| g(u,a_m,\e_m) - g\Bigl(u,a_m,\frac{i}{m}\Bigr)\right|
\end{align*}
Therefore, by Rouche's Theorem, the equation
\begin{equation} \label{equation a_m eps}
  \begin{split}
   a_m \, g(u,a_m,\e_m) = (e^{2 \pi m \e_m} - 1)
   &+ a_m \, I_{(m)}(u,\e_m) + \\
   &+ a^2_m \, G_{(m)}(u,a_m,\e_m) = 0
   \end{split}
\end{equation}
has as many solutions as
\begin{equation} \label{equation a_m im}
a_m \, g\Bigl(u,a_m,\frac{i}{m}\Bigr) = a_m \,
I_{(m)}\Bigl(u,\frac{i}{m}\Bigr) + a_m^2 \,
G_{(m)}\Bigl(u,a_m,\frac{i}{m}\Bigr) = 0.
\end{equation}
By Fact 1, equation (\ref{equation a_m im}) has $m$ roots
$u_1(a_m),...,u_m(a_m)$ contained in $D_{r_1}(0)$. For that
reason, equation (\ref{equation a_m eps}) has $m$ solutions
contained in $D_{r_1}(0).$ Let us denote them by $u_1(a_m,\e_m),$
...,$u_m(a_m,\e_m).$ As it was established earlier, none of them
is zero. For simplicity, let $u_j = u_j(a_0,\e_0),$ where $j =
1,..,m.$

By Fact 2, for $1 \leq k \leq m-1$ and for $u \in D_{r_1}(0)$,
$$|e^{2 \pi k \e_m} - 1| \geq \mu(a_m) > \mu > |a_m| \, \bigl|I_{(k)}(u,\e_m) + a_m G_{(k)}(u,a_m,\e_m)\bigr|.$$
Having in mind that $u_j \in D_{r_0}(0)$ and each of them is
nonzero for $j=1,..,m$, we estimate
\begin{align*}
\bigl|P_{a_m,\e_m}^k(u_j)-u_j\bigr| =& |u_j| \, \bigl|(e^{2 \pi k
\e_m} - 1) + a_m \, I_{(k)}(u_j,\e_m) \\
&+ a^2_m \, G_{(k)}(u_j,a_m,\e_m)\bigr| \geq |u_j| \, \bigl( |e^{2
\pi k \e_m} - 1| \\ &- |a_m| \, \bigl|I_{(k)}(u_j,\e_m) + a^2_m \,
G_{(k)}(u_j,a_m,\e_m)\bigr|\bigr) > 0.
\end{align*}
For that reason, $P^k_{a_m,\e_m}(u_j) \neq u_j$ for $1 \leq k \leq
m-1$. Hence, the orbit $u_1$,...,$u_m$ consists of different
points and therefore is periodic of period $m$ in $D_{r_1}(0).$
\end{proof}

%\bibliographystyle{ams}
%\bibliography{Article2010_Version5}

\end{document}